\documentclass[lettersize,journal]{IEEEtran}
\usepackage{amsmath,amsfonts}
\usepackage{amsthm}
\usepackage{algorithmic}
\usepackage{algorithm}
\usepackage{array}
\usepackage[caption=false,font=normalsize,labelfont=sf,textfont=sf]{subfig}
\usepackage{textcomp}
\usepackage{stfloats}
\usepackage{url}
\usepackage{verbatim}
\usepackage{graphicx}
\usepackage{cite}
\usepackage{booktabs}
\newtheorem{thm}{Theorem}
\newtheorem{lem}{Lemma}
\newtheorem{exm}{Example}
\newtheorem{asm}{Assumption}
\newtheorem{defin}{Definition}
\newtheorem{rem}{Remark}
\begin{document}

\title{on a stochastic column-block bregman method for nonlinear systems}

\author{Wendi Bao, Naiyu Jiang, Xing Lili, Li Weiguo\\~\IEEEmembership{College of Science,	China University of Petroleum, Qingdao 266580, P.R. China}}

\markboth{TRANSACTIONS ON IMAGE PROCESSING}%
{Naiyu Jiang, Wendi Bao \MakeLowercase{\textit{et al.}}:The stochastic column-block nonlinear Bregman method for solving nonlinear systems of equations}


\maketitle

\begin{abstract}
In this paper, we proposed a stochastic column-block nonlinear Bregman method with uniform probability to find the sparse solution of nonlinear systems of equations. Under certain assumption we conclude the upper bound for the convergence rates of this new method. Numerical experiments demonstrate that the efficiency of the new method.
\end{abstract}


\section{Introduction}
\IEEEPARstart{F}{inding} the sparse solution of nonlinear systems of equations
\begin{equation}\label{eq:1.1}
	f(x)=0,
\end{equation}
,where $x\in\mathbb{R}^n$ and $f(x):\mathbb{R}^n\to\mathbb{R}^m$, has many applications, such as compressed sensing \cite{Needell12}, image processing \cite{Dong16}, deep learning \cite{Dettmers19}.
This problem can be reformulated as the nonlinear constrained optimization problem
\begin{equation}\label{eq:1.2}
	\mathop{\min}_{x\in \mathcal{D}}\varphi(x),\quad s.t. \quad f(x)=0.
\end{equation}	
where $\varphi$ is a convex and nonsmooth function which is also called the sparsity inducing function. We denote $\varphi:\mathbb{R}^n\to \overline{\mathbb{R}}:\mathbb{R}\cup\{+\infty\}$ is convex with $\mathrm{dom }\varphi=\{x\in \mathbb{R}^n:\varphi(x)<\infty\}\neq\emptyset$. 

The nonlinear Kaczmarz (NK) method \cite{Wang21} was proposed to solve the nonlinear systems of equations (\ref{eq:1.1}), which only one equation was used at each iteration. Since its efficient and less storage, many variations has been developed, for instance, the maximum residual nonlinear Kaczmarz (MRNK) method \cite{Zhang23} , the pseudoinverse-free greedy block nonlinear method \cite{Lv24}, and the residual-based weighted nonlinear Kaczmarz (RBWNK) method \cite{Ye24}. Besides, motivated by the doubly stochastic block Gauss-Seidel (DSBGS) method Bao et al. \cite{Bao25} proposed the doubly stochastic block method for nonlinear equations (DSBN) method, which unifies some stochastic row-block and column-block methods. And the stochastic column-block method for nonlinear equations was developed, which use uniform probability to select columns.   

The Bregman method \cite{Yin08} was proposed to solve to solve the regularized version of the basis pursuit problem 
\begin{equation*}
	\mathop{\min}_{x\in\mathbb{R}^n}\lambda\|x\|_1+\frac{1}{2}\|x\|_2^2\quad s.t.\quad Ax=b.
\end{equation*} 
In \cite{Lorenz14_1} \cite{Lorenz14_2}, Kaczmarz method was introduced into the Bregman method obtain the sparse Kaczmarz method. Sch\"{o}pfer et al. \cite{Schopfer19} developed the randomized version of the sparse Kaczmarz (RSK) method and gave its convergence analysis. In \cite{Zhang22}, a weighted randomized sparse Kaczmarz (WRaSK) method was developed, which selects the $i$-th projection hyperplane with probability proportional to $|\langle a_i,x_k\rangle-b_i|^p$, where $0<p<\infty$, for possible acceleration. Yuan et. al \cite{Yuan22} used sketch-and-project to Bregman method and proposed the sketched Bregman projection (SBP) method. They generalized the concept of adaptive sampling to SBP methods, and showed how the single-step progress—measured by the Bregman distance—depends directly on a sketched loss function. Tondji et al. \cite{Tondji23} introduced a parallel (mini batch) version of RSK based on averaging several Kaczmarz steps and proposed Randomized sparse Kaczmarz with averaging (RSKA). Inspired of RSKA method Xiao et. al \cite{Xiao24} introduced a flexible greedy selection strategy, which is adopted to adaptively choose the subset corresponding to the subvectors of the residuals with relatively large norms and developed the greedy average block sparse Kaczmarz (GABSK) method. Niu et. al \cite{Niu25a} proposed a fast block sparse Kaczmarz (FBSK) method based on the Motzkin criterion and it converges linearly to the sparse solutions of the linear systems.

While solving the sparse solution of nonlinear systems of equations is still in the developmental stage. By introducing the Bregman projection to nonlinear Kaczmarz (NK) method \cite{Wang21}, Gower \cite{Gower24} proposed the nonlinear Bregman-Kaczmarz (NBK) method and relaxed nonlinear Bregman Kaczmarz (rNBK) method. To accelerate the speed of convergence, Xiao et al. proposed the averaging block nonlinear Bregman Kaczmarz method which is the row-block type of the NBK method.

To accelerate the convergence of the NBK method, we consider sampling some columns of the gradient with uniform probability at each iteration. We assign a certain weight to the values computed for each column, sum these weighted values, multiply by the corresponding step size to perform the iteration, and propose the stochastic column-block nonlinear Bregman (SCBNB) method to find the sparse solution of nonlinear equations.

The notations in this paper are expressed as follows. For any matrix $A\in \mathbb{R}^{m\times n}$, $\| \cdot\|$, $\sigma_{\min}$, $\sigma_{\max}$ denote the Euclidean norm, the maximum and minimum nonzero singular values of a matrix $A$, respectively. Let $\nabla f(x)$ be the Jacobi matrix of $f(x)$ and $\nabla f_{:,j}(x)$ is the $j$th column of the Jacobi matrix $\nabla f(x)$.

The rest of this paper is organized as follows. In Section \ref{sec2}, the SCBNB method is proposed and the convergence of SCBNB method is analyzed. We perform some numerical experiments in Section \ref{sec3}, which verify the efficiency of the SCBGD method and SCBNB method. Finally, we have the conclusions in Section \ref{sec4}.

\section{Stochastic column-block nonlinear Bregman method}
\label{sec2}

In this section, we consider solving (\ref{eq:1.2}). We first introduce some preliminaries. 

\begin{defin}
	The subdifferential of $\varphi$ at a point $x\in \mathrm{dom }\varphi$ is defined as 
	\begin{equation*}
		\resizebox{1\hsize}{!}{$\partial\varphi(x)=\{x^*\in \mathbb{R}^n:\varphi(x)+\langle x^*,y-x\rangle\leq\varphi(y) \quad\text{for all }\quad y\in \mathrm{dom } \varphi\}$}
	\end{equation*} 
	$x^*\in\partial\varphi(x)$ is called a subgradient of $\varphi$ at $x$.
\end{defin}
\begin{defin}
	The Bregman distance $D_{\varphi}^{x^*}(x,y)$ between $x,y\in dom \varphi$ with respect to $\varphi$ and a subgradient $x^*\in \partial\varphi(x)$ is defined as 
	\begin{equation*}
		D_{\varphi}^{x^*}(x,y)=\varphi(y)-\varphi(x)-\langle x^*,y-x\rangle
	\end{equation*}
\end{defin}
\begin{defin}
	Denote $\mathcal{B}$ be a nonempty convex set, $\mathcal{B}\cap\mathrm{dom}\varphi\neq\emptyset$, $x\in\mathrm{dom}\varphi$ and $x^*\in\partial\varphi$. The Bregman projection of $x$ onto $\mathcal{B}$ with respect to $\varphi$ and $x^*$ is the point $\Pi_{\mathcal{B}}^{x^*}(x)\in\mathcal{B}\cap\mathrm{dom}\varphi$ such that
	\begin{equation*}
		\Pi_{\mathcal{B}}^{x^*}(x)=arg\mathop{\min}_{y\in\mathcal{B}}D_\varphi^{x^*}(x,y).
	\end{equation*}
\end{defin}

\begin{lem}\label{lem:D} 
	If $\varphi$ is a $\gamma$-strongly convex function, then for $\forall x\in \mathcal{D}$ the following inequality holds
	\begin{align}\label{eq:D}
		D_{\varphi}^{x_{k+1}^*}(x_{k+1},\hat{x})
		&\leq D_{\varphi}^{x_{k}^*}(x_{k},\hat{x})+\langle x_{k+1}^*-x_{k}^*,x_k-\hat{x}\rangle\notag\\
		&\quad+\frac{1}{2\gamma}\|x_{k+1}^*-x_{k}^*\|_2^2.
	\end{align}
\end{lem}

Now we recall the nonlinear Bregman-Kaczmarz method, which take the Bregman projection of $x_k$ onto the set $H_k$, that is to solve 
\begin{equation*}
	x_{k+1}=arg\mathop{\min}_{x\in \mathcal{D}}D_{\varphi}^{x^*}(x_k,x),\quad s.t.\quad x\in H_k,
\end{equation*}
with 
\begin{align*}
	H_k:
	&=\{x\in\mathbb{R}^n:f_{i_k}(x_k)+\langle \nabla f_{i_k}(x_k),x-x_k\rangle=0\}\\
	&=H(\nabla f_{i_k}(x_k),\beta_k),
\end{align*} 
where $\beta_k=\langle \nabla f_{i_k}(x_k),x_k\rangle-f_{i_k}(x_k)$,  $i_{k}\in\{1,2,...,n\}$ and $x_k^*$ is the subgradient $\partial\varphi(x_k)$. 
If 
\begin{equation}\label{Hk dom varphi}
	H_k\cap\mathrm{dom}\partial\varphi\neq\emptyset,
\end{equation}
taking the Bregman projection of $x_k$ onto the set $H_k$, we can get the next iteration $x^*_{k+1}=x_k^*-t_{k,\varphi}\nabla f_{i_k}(x_k)$ and $x_{k+1}=\nabla\varphi^*(x_{k+1}^*)$ with 
\begin{equation*}
	t_{k,\varphi}\in arg\mathop{\min}_{t\in \mathbb{R}}\varphi^*(x_k^*-t\nabla f_{i_k}(x_k))+\beta_kt.
\end{equation*}
If (\ref{Hk dom varphi}) is not fulfilled then let $x^*_{k+1}=x_k^*-t_{k,\gamma}\nabla f_{i_k}(x_k)$ with the Polyak-like step size $t_{k,\gamma}=\gamma\frac{f_{i_k}(x_k)}{\|\nabla f_{i_k}(x_k)\|_*^2}$ with some norm $\|\cdot\|_*$ and some constant $\gamma>0$.
The algorithm of the NBK method is shown in Algorithm \ref{alg: NBK}. Further considering the relaxed version, we obtain the rNBK method shown in Algorithm \ref{alg: rNBK}.

\begin{algorithm}[H]
	\caption{NBK: Nonlinear Bregman-Kaczmarz method}\label{alg: NBK}
	\begin{algorithmic}
		\STATE 
		\STATE \textbf{Parameters:} $\gamma>0$, probabilities $p_i>0$ for $i=1,...,m$
		\STATE \textbf{Initialization:} Choose $x^*_0\in \mathbb{R}^n$,$x_0=\nabla\varphi^*(x_0^*)$
		\STATE \textbf{for} $k=1,2,\cdots $
		\STATE \hspace{0.5cm} Pick $i_k \in \{1,...,n\}$ with the probabilities $p_1,...p_m$
		\STATE \hspace{0.5cm} \textbf{if}{$f_{i_k}(x_k)\neq0$ and $\nabla f_{i_k}(x_k)\neq0$}
		\STATE \hspace{1cm} Set $\beta_k=\langle \nabla f_{i_k}(x_k),x_k\rangle-f_{i_k}(x_k)$
		\STATE \hspace{1cm} \textbf{if} $H_k\cap\mathrm{dom}\partial\varphi\neq\emptyset$
		\STATE \hspace{1.5cm} Find \resizebox{0.7\hsize}{!}{$t_k: t_k\in arg\mathop{\min}_{r\in\mathbb{R}}\varphi^*(x_k^*-t\nabla f_{i_k}(x_k))+t\beta_k$}
		\STATE \hspace{1cm} \textbf{else}
		\STATE \hspace{1.5cm} Set $t_k=\gamma\frac{f_{i_k}(x_k)}{\|\nabla f_{i_k}(x_k)\|_*^2}$
		\STATE \hspace{1.5cm} Compute $x^*_{k+1}=x^*_k-t_k\nabla f_{:,\xi_k}(x_k)^T$
		\STATE \hspace{1.6cm}Compute $x_{k+1}=\nabla\varphi^*(x_{k+1}^*)$
		\STATE \hspace{1cm}\textbf{endif}
		\STATE \hspace{0.5cm}\textbf{endif}
		\STATE \textbf{endfor}
		\STATE \textbf{return:} last iterate $x_{k}$
	\end{algorithmic}
\end{algorithm}

\begin{algorithm}[H]
	\caption{rNBK: Relaxed Nonlinear Bregman-Kaczmarz method}\label{alg: rNBK}
	\begin{algorithmic}
		\STATE 
		\STATE \textbf{Parameters:} $\gamma>0$, probabilities $p_i>0$ for $i=1,...,m$
		\STATE \textbf{Initialization:} Choose $x^*_0\in \mathbb{R}^n$,$x_0=\nabla\varphi^*(x_0^*)$
		\STATE \textbf{for} $k=1,2,\cdots $
		\STATE \hspace{0.5cm} Pick $i_k \in \{1,...,n\}$ with the probabilities $p_1,...p_m$
		\STATE \hspace{0.5cm} \textbf{if}{$f_{i_k}(x_k)\neq0$ and $\nabla f_{i_k}(x_k)\neq0$}
		\STATE \hspace{1cm} Compute $x^*_{k+1}=x^*_k-\gamma\frac{f_{i_k}(x_k)}{\|\nabla f_{i_k}(x_k)\|_*^2}\nabla f_{i_k}(x_k)$
		\STATE \hspace{1.1cm}Compute $x_{k+1}=\nabla\varphi^*(x_{k+1}^*)$
		\STATE \hspace{0.5cm}\textbf{endif}
		\STATE \textbf{endfor}
		\STATE \textbf{return:} last iterate $x_{k}$
	\end{algorithmic}
\end{algorithm}

The above methods are row type methods, next we consider the column-type method to solve the sparse solution of nonlinear equations. 

Since the update rule of $x_k^*$ in the Bregman-Kaczmarz method is similar to that in the Kaczmarz method, but with an additional convex parameter $\gamma$, we consider using the update rule of the SCBGD method to update $x_k^*$. The basic steps are similar to the nonlinear Bregman-Kaczmarz method, but when updating $x_{k+1}^*$, we refer to the SCBGD method to choose the direction of the update as $I_{:,\xi_k}\nabla f_{:,\xi_k}(x_k)^Tf(x_k)$ instead of $\nabla f_{i_k}(x_k)$ such that the elements of $x_k^*$ are only partially updated. Then we have
\begin{align*}
	x_{k+1}^*&=x_k^*-t_kI_{:,\xi_k}\nabla f_{:,\xi_k}(x_k)^Tf(x_k),\\
	x_{k+1}&=\nabla \varphi^*(x_{k+1}^*).
\end{align*}
The step size we refer to the SCBGD method, then we can obtain the SCBNB method to solve (\ref{eq:1.2}). The detailed steps of the SCBNB method we can obtain from Algorithm \ref{alg: SCBNB}.

	\begin{algorithm}[H]
		\caption{SCBNB: Stochastic Column-Block Nonlinear Bregman method}\label{alg: SCBNB}
		\begin{algorithmic}
			\STATE 
			\STATE \textbf{parameters:} $\delta\in(0,2)$, $q>0$, $\gamma>0$
			\STATE \textbf{initialization:} Choose $x^*_0\in \mathbb{R}^n$,$x_0=\nabla\varphi^*(x_0^*)$
			\STATE \textbf{for} $k=1,2,\cdots $
			\STATE \hspace{0.5cm} Pick $\xi_k \in \left\{J_1,J_2,\cdots,J_\tau\right\}$ with uniform probability
			\STATE \hspace{0.5cm} Compute $x^*_{k+1}=x^*_k-$
			\STATE \hspace{1cm}$\delta\frac{\gamma\|\nabla f_{:,\xi_k}(x_k)^Tf(x_k)\|_2^2}{\|\nabla f_{:,\xi_k}(x_k)\nabla f_{:,\xi_k}(x_k)^Tf(x_k)\|_2^2}I_{:,\xi_k}\nabla f_{:,\xi_k}(x_k)^Tf(x_k)$
			\STATE \hspace{0.5cm}Compute $x_{k+1}=\nabla\varphi^*(x_{k+1}^*)$
			\STATE \textbf{endfor}
			\STATE \textbf{return:} last iterate $x_{k}$
		\end{algorithmic}
	\end{algorithm}
	
	Next we analyze the convergence property of SCBNB on Bregman projections. Assume that $\hat{x}$ is the solution of problem (\ref{eq:1.2}).
	
	\begin{defin}
		\cite{Kaltenbacher08} Let $f(x)$ be a continuously differentiable mapping. Then $f(x)$ satisfies
		tangential cone condition in $\Omega\subset \mathbb{R}^n$ if there exists $0<\eta<1$ such that for any $x_1,x_2\in\Omega$,
		\begin{equation*}
			|f(x_1)-f(x_2)-\nabla f(x_1)(x_1-x_2)|\leq\eta|f(x_1)-f(x_2)|
		\end{equation*}   
	\end{defin}
	
	\begin{asm}\label{asm1}
		
		i) $f(x):\mathcal{D}\subseteq\mathbb{R}^n\to\mathbb{R}^m$ on a bounded closed $\mathcal{D}$ is derivable and the derivative of $f(x)$ is continuous in $\mathcal{D}$.
		
		ii) $f(x)$ satisfies the tangential cone condition. 
				
		iii) $\varphi:R^n\to R$ is proper, convex and lower semicontinuous.
		
		iv) $\nabla f(\hat{x})$ is full column rank matrix.	
	\end{asm}
	
		%
		%
	\begin{lem}
		\label{lem:1+eta^2}
		{\rm\cite{Zhang24}}
		If the nonlinear function f satisfies the local tangential cone condition in $\mathcal{D}$, then for $\forall x_1, x_2 \in D$  and an index subset $\tau\subseteq [m]$, we have
		\begin{equation}\label{eq:block ltcc}
			\|f_\tau(x_1)-f_\tau(x_2)\|_2^2\geq\frac{1}{(1+\eta)^2}\|f_\tau^\prime(x_1)(x_1-x_2)\|_2^2.
		\end{equation}
	\end{lem}
	
	\begin{lem}\label{lem:D<=(x*-y*)(x-y)}
		Let $\varphi(x)$ be $\gamma$-strongly convex. For any $x,y\in\mathbb{R}^n$ and subgradients $x^*\in\partial\varphi(x)$, $y^*\in\partial\varphi(y)$, we have
		\begin{equation*}
			\frac{\gamma}{2}\|x-y\|_2^2\leq D_\varphi^{x^*}(x,y)\leq\langle x^*-y^*,x-y\rangle\leq\|x^*-y^*\|_2\|x-y\|_2.
		\end{equation*}
	\end{lem}
		\begin{lem}\label{lem:smooth}
		If $\varphi$ is a convex and lower semicontinuous, then the following statements are equivalent:
		
		i) $\varphi$ is M-smooth with respect to a norm $\|\cdot\|_2$,
		
		ii) $\varphi(y)\leq\varphi(x)+\langle\nabla\varphi(x),y-x\rangle+\frac{M}{2}\|x-y\|_2^2$ for all $x,y\in \mathbb{R}^n$,
		
		iii) $\langle\nabla\varphi(y)-\nabla\varphi(x),y-x\rangle\leq M\|x-y\|_2^2$ for all $x,y\in \mathbb{R}^n$.

	\end{lem}
	
	\begin{lem}\label{lem:sigma}
		\cite{Jiang25}Assume that Assumption \ref{asm1} i) holds true. Then there exists $\underline{\sigma}$ and $\overline{\sigma}$ such that
		\begin{equation*}
			\underset{x\in\mathcal{D}}{\inf}\ \sigma_{\min}(\nabla f(x))=\underline{\sigma}>0,\
			\underset{x\in\mathcal{D}}{\sup}\ \sigma_{\max}(\nabla f(x))=\overline{\sigma}<\infty.
		\end{equation*} 
		\begin{proof}
			Since $\nabla f(x)$ is continuous in $\mathcal{D}$ and the singular value of $\nabla f(x)$ is a continuous function of $\nabla f(x)$, we can get the singular value of $\nabla f(x)$ is a continuous function of $x$.
			What's more, $\mathcal{D}$ is a bounded closed and $\sigma_{\min}(x_k)$ is continuous function of $x$, there exist a point $\underline{x}\in\mathcal{D}$ such that $\underline{\sigma}=\sigma_{\min}(\underline{x})=\underset{x\in\mathcal{D}}{\inf}\ \sigma_{\min}(x_k)$. Then
			\begin{equation*}
				\sigma_{\min}(\nabla f(x_k))\geq\underset{x\in\mathcal{D}}{\inf}\ \sigma_{\min}(\nabla f(x_k))=\underline{\sigma}>0.
			\end{equation*}
			Similarly there exists a point $\overline{x}\in\mathcal{D}$ such that $\overline{\sigma}=\sigma_{\max}(\overline{x})=\underset{x\in\mathcal{D}}{\sup}\ \sigma_{\max}(x_k)<\infty$. Then
			\begin{equation*}
				\sigma_{max}(\nabla f(x_k))\leq\underset{x\in\mathcal{D}}{\sup}\ \sigma_{\max}(\nabla f(x_k))=\overline{\sigma}<\infty.
			\end{equation*}
		\end{proof}
	\end{lem}
	
	\begin{thm}\label{thm1}
		Assume Assumption \ref{asm1} holds and $\varphi$ is $\gamma$-strongly convex and M-smooth with respect to $\|\cdot\|_2$ and $\underline{\sigma}^4-\overline{\sigma}^4+\underline{\sigma}^2\overline{\sigma}^2>0$, and $a+b>0$, where $a=\underline{\sigma}^4>0$ and $b=-\overline{\sigma}^4+\underline{\sigma}^2\overline{\sigma}^2\leq0$. Assume that when $b=0$, there is $0<\eta<1$ and  when $b\neq0$, there is $0<\eta<\min\{1,\frac{a-b}{a+b}-2\frac{\sqrt{|ab|}}{a+b}\}$. If $0<\delta\leq min\{2,Q\}$ where $Q=2\left(\frac{\underline{\sigma}^4}{\overline{\sigma}^2}(1-\eta)-\frac{(\overline{\sigma}^2-\underline{\sigma}^2)(1+\eta)^2}{\overline{\sigma}^2(1-\eta)}\right)$, then the sequence $\{x_k\}$ and $\{x_k^*\}$ generated by Algorithm \ref{alg: SCBNB} satisfies
		\begin{equation}\label{eq: D_k<D_0}
			\mathbf{E}[D_{\varphi}^{x_{k+1}^*}(x_{k+1},\hat{x})]\leq (1-c)^{k+1}\mathbf{E}[D_{\varphi}^{x_{0}^*}(x_{0},\hat{x})],
		\end{equation}
		where \resizebox{0.9\hsize}{!}{$c=\frac{1}{M}\left[\frac{2\gamma\underline{\sigma}^4(1-\eta)^2-2\gamma(\overline{\sigma}^2-\underline{\sigma}^2)(1+\eta)^2\overline{\sigma}^2}{2\tau(1+\eta)^2\overline{\sigma}^2\underline{\sigma}^2(1-\eta)}\delta-\frac{\gamma\underline{\sigma}^4(1-\eta)}{2\tau(1+\eta)^2\overline{\sigma}^2\underline{\sigma}^2(1-\eta)}\delta^2\right]$}.
		Further, if $x_0^*=x_0=0$, the error between $x_k$ and the solution $\hat{x}$ satisfies:
		\begin{equation}\label{eq: x_k-x*}
			\|x_k-\hat{x}\|_2^2\leq \frac{2}{\gamma} (1-c)^k\varphi(\hat{x}).
		\end{equation}
	\end{thm}
	
	\begin{rem}
		If $\varphi(x)=\|x\|_1+\frac{1}{2}\|x\|_2^2$, which is a $1-$strongly convex function, then from Theorem \ref{thm1} we have
		\begin{equation*}
			\mathbf{E}[D_{\varphi}^{x_{k}^*}(x_{k+1},\hat{x})]\leq (1-c)^k\mathbf{E}[D_{\varphi}^{x_{0}^*}(x_{0},\hat{x})],
		\end{equation*}
		where \resizebox{0.9\hsize}{!}{$c=\frac{1}{M}\left[\frac{2\gamma\underline{\sigma}^4(1-\eta)^2-2\gamma(\overline{\sigma}^2-\underline{\sigma}^2)(1+\eta)^2\overline{\sigma}^2}{2\tau(1+\eta)^2\overline{\sigma}^2\underline{\sigma}^2(1-\eta)}\delta-\frac{\gamma\underline{\sigma}^4(1-\eta)}{2\tau(1+\eta)^2\overline{\sigma}^2\underline{\sigma}^2(1-\eta)}\delta^2\right]$}.
		Further, if $x_0^*=x_0=0$, the error between $x_k$ and the solution $\hat{x}$ satisfies:
		\begin{equation*}
			\|x_k-\hat{x}\|_2^2\leq 2(1-c)^k\left(\|\hat{x}\|_1+\frac{1}{2}\|\hat{x}\|_2^2\right).
		\end{equation*}
	\end{rem}

	\section{Numerical experiments}
\label{sec3}

In this section we present some numerical experiments to verify the effectiveness of the SCBNB method for solving the sparse solution of nonlinear systems of equations. We compare the SCBNB method with the nonlinear Bregman-Kaczmarz (NBK) method, the relaxed nonlinear Bregman-Kaczmarz rNBK method, maximum residual nonlinear Bregman-Kaczmarz (MRNBK) method and the relaxed maximum residual nonlinear Bregman-Kaczmarz )(rMRNBK) method. The index $i$ for NBK and rNBK method is chosen by the probability criterion $\frac{|r_i|^2}{\|r\|_2^2}$. The index $i$ for MRNBK and rMRNBK method is chosen by $i=arg\mathop{\max}_{1\leq i\leq m} |r_i|^2$. We use the number of iteration steps (IT) and the elapsed computing time in seconds (CPU) to measure the effectiveness of the methods. The iteration termination criterion is set to the relative residual $\frac{\|f(x_k)\|_2^2}{\|f(x_0)\|_2^2}\leq10^{-6}$ or the IT exceed $1000$. In tables, the "$--$" defines the IT exceeding 1000.

All experiments are performed using MATLAB (version R2021a) on a personal computer with a 2.20 GHz central processing unit  (13th Gen Inte l(R) Core (TM) i9-13900HX), 32.0 GB memory and Windows operating system (64 bit Windows 11).

We consider solving the sparse solution of quadratic equations
\begin{equation*}
	f_i(x)=\frac{1}{2}\langle x,A^{(i)}x\rangle+\langle b^{(i)},x\rangle+c^{(i)}=0 \quad 1\leq i\leq m
\end{equation*}
with $A^{(i)}\in\mathbb{R}^{n\times n}$, $b^{(i)}\in\mathbb{R}^n$, $c^{(i)}\in\mathbb{R}$.

In this experiment, $A^{(i)}$ has two alternative strategies which is presented later and $b^{(i)}$ is sampled randomly with entries form the standard normal distribution.
The sparsity of the sparse solution $\hat{x}$ is denoted by '$sp$' and nonzero entries of original signal $\hat{x}$ is generated from the standard normal distribution. $b^{(i)}$ and $\hat{x}$ can be generated by MATLAB function $rand(n,1)$ and $sprandn(n,1,sp)$, respectively. Set 
\begin{equation*}
	c^{(i)}=-\left(\frac{1}{2}\langle x,A^{(i)}x\rangle+\langle b^{(i)},x\rangle\right).
\end{equation*}

We use $\varphi(x)=\lambda\|x\|_1+\frac{1}{2}\|x\|_2^2$ and set $\lambda=2$ the initial subgradient $x_0^*$ is generated from the standard normal distribution. Then the initial vector $x_0=\nabla\varphi^*(x_0^*)=S_\lambda(x_0^*)$, where
\begin{equation*}
	S_\lambda(x)=
	\begin{cases}
		x+\lambda, &x<-\lambda\\
		0, &|x|\leq -\lambda\\
		x-\lambda, &x>\lambda
	\end{cases}.
\end{equation*} 

\begin{exm}\label{ex:rand}
	$A^{(i)}$ is sampled from the standard normal distribution.
\end{exm} 
It can be generated by MATLAB function $randn(n,n)$. In Table \ref{tab:A=randn0.1} and Table \ref{tab:A=randn0.05} we compare the IT and CPU of the new method against the existing methods on different dimensions of the problem for $sp=0.1$ and $sp=0.05$ of $\hat{x}$, respectively. For the SCBNB method we take $q=25$ in the case of $m=200$,  $n=100$ and take $q=30$ in other cases. We find the IT and CPU of the SCBNB method is the best than other methods in most cases, which indicates the effectiveness of the SCBNB method for solving the sparse solution of nonlinear systems of equations. Figure \ref{figure A=randn0.1} and Figure \ref{figure A=randn0.05} shows for the NBK, rNBK, MRNBK and rMRNBK methods as the dimensionality increases, the convergence of a gradually slows down and become difficult to get sparse solution. While the SCBNB method still has respectable convergence rates as the dimensionality increases. The original signal and the recovered signals by the five methods are depicted in Figure \ref{figure randn solution}, respectively. We find the SCBNB method has the best performance to recovery the signal. 

In addition, let $x_0=x_0^*=0$ and $sp=0.05:0.05:0.15$. Set $m=100$ and $n$ ranging from $50$ to $150$ in increments of $20$. We asses the success rate if the NBK, rNBK, MRNBK, rMRNBK, and SCBNB methods in signal recovery. We conduct 100 experiments with different settings of $A$, $b$ and $\hat{x}$. The relative residuals can satisfy $\frac{\|f(x_k)\|_2^2}{\|f(x_0)\|_2^2}\leq10^{-6}$ within 3000 iterations, which is regarded as the success of the experiment. Similar as \cite{Niu25b} we visualize and compare the success rates achieved by the above five methods in Figure \ref{figure randn success}. It shows that the success rate of the SCBNB method has the best performance among the five methods. And NBK and MRNBK methods with exact step have better success rates than rNBK and rMRNBK methods which use inexact step. And after fixing $m$, the success rate decreases as $sp$ increases as well as $n$ increases.

\begin{table*}[!t]
	\caption{IT and CPU of the NBK, \textit{r}NBK, MRNBK, \textit{r}MRNBK and SCBNB methods for $A^{(i)}$ is sampled from the standard normal distribution with $sp=0.1$\label{tab:A=randn0.1}}
	\centering
	\begin{tabular}{cclllllllllll}
		\toprule
		$m$&  $n$& $sp$ & \multicolumn{2}{c}{NBK}& \multicolumn{2}{c}{rNBK}& \multicolumn{2}{c}{MRNBK}& \multicolumn{2}{c}{rMRNBK}& \multicolumn{2}{c}{SCBNB}\\
		& & &  IT&CPU& IT&  CPU& IT&CPU& IT& CPU& IT&CPU\\
		\midrule
		$200$&
		$100$& $0.1$& $681$&$0.7661$&$--$&$--$& $354$&$\mathbf{0.3887}$&$--$& $--$& $137$&$0.4892$\\
		$300$&$150$& $0.1$& $--$&$--$&$--$&$--$& $--$&$--$&$--$& $--$& $155$&$\mathbf{2.1753}$\\
		$400$&$200$& $0.1$& $--$&$--$&$--$&$--$& $--$&$--$&$--$& $--$& $446$&$\mathbf{14.0919}$\\
		$500$&$250$& $0.1$& $--$& $--$& $--$& $--$& $--$& $--$& $--$& $--$& $593$& $\mathbf{36.0715}$ \\\bottomrule\end{tabular}
\end{table*}

\begin{table*}[!t]
	\centering
	\caption{IT and CPU of the NBK, \textit{r}NBK, MRNBK, \textit{r}MRNBK and SCBNB methods for $A^{(i)}$ is sampled from the standard normal distribution with $sp=0.05$\label{tab:A=randn0.05}}
	\begin{tabular}{cclllllllllll}
		\toprule
		$m$&  $n$& $sp$ & \multicolumn{2}{c}{NBK}& \multicolumn{2}{c}{rNBK}& \multicolumn{2}{c}{MRNBK}& \multicolumn{2}{c}{rMRNBK}& \multicolumn{2}{c}{SCBNB}\\
		& & &  IT&CPU& IT&  CPU& IT&CPU& IT& CPU& IT&CPU\\
		\midrule
		$200$&
		$100$& $0.05$& $490$&$0.5538$&$--$&$--$& $399$&$0.4397$&$--$& $--$& $97$&$\mathbf{0.3469}$\\
		$300$&$150$& $0.05$& $511$&$3.1988$&$--$&$--$& $294$&$1.9988$&$--$& $--$& $126$&$\mathbf{1.7529}$\\
		$400$&$200$& $0.05$& $741$&$15.7613$&$--$&$--$& $604$&$14.0252$&$--$& $--$& $250$&$\mathbf{8.0981}$\\
		$500$&$250$& $0.05$& $--$&$--$&$--$&$--$& $672$&$29.4431$&$--$& $--$& $413$&$\mathbf{28.0365}$\\\bottomrule\end{tabular}
\end{table*}

\begin{figure*}[!t]
	\centering
	\subfloat[$m=200,n=100$]{\includegraphics[width=1.7in]{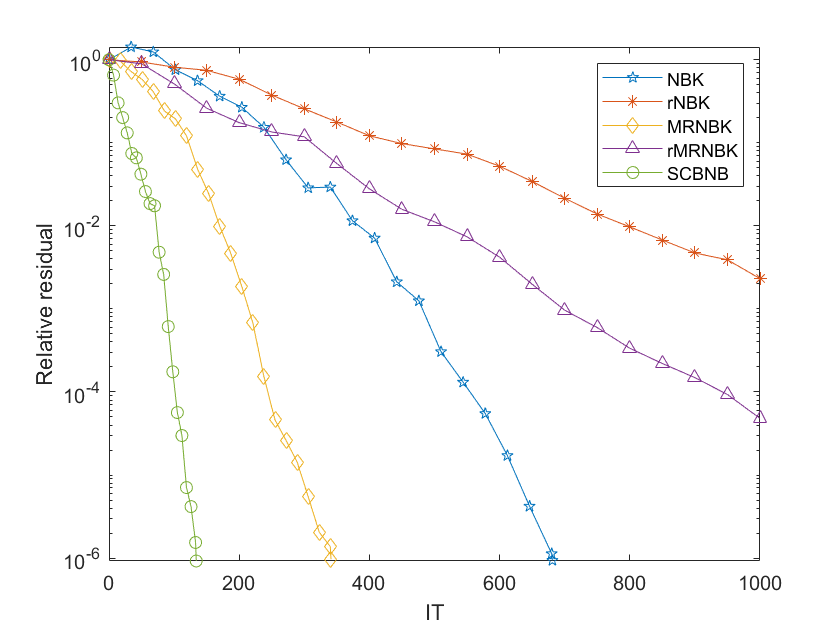}}
	\subfloat[$m=300,n=150$]{\includegraphics[width=1.7in]{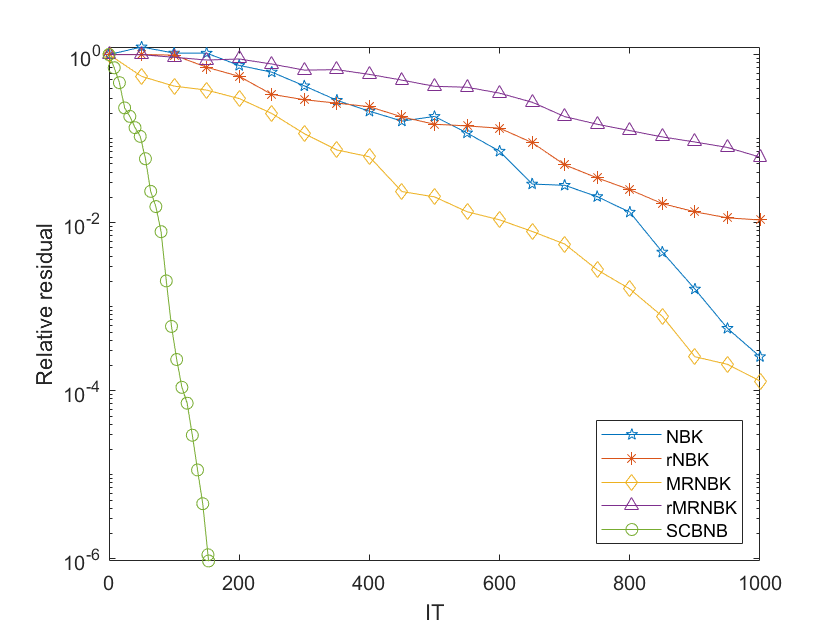}}
	\subfloat[$m=400,n=200$]{\includegraphics[width=1.7in]{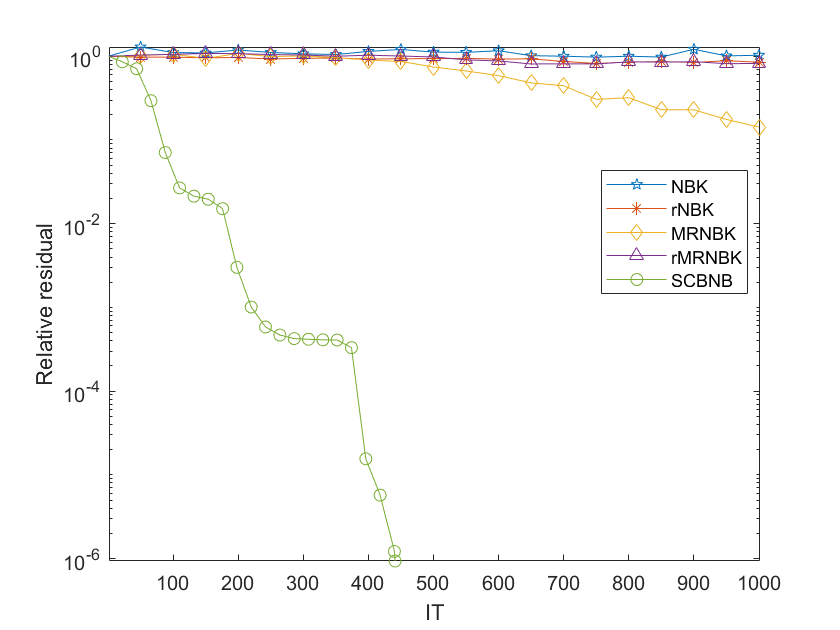}}
	\subfloat[$m=500,n=250$]{\includegraphics[width=1.7in]{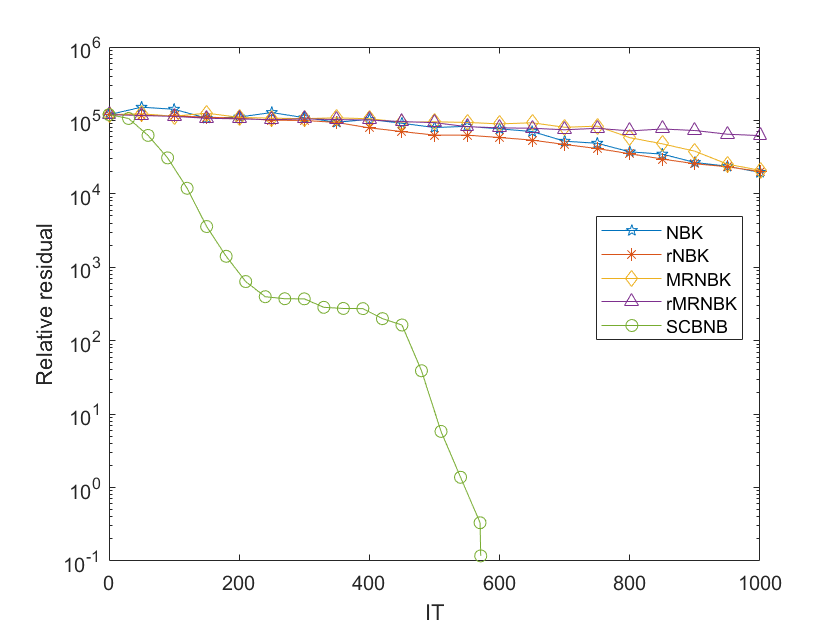}}
	\caption{Results for $A^{(i)}$ which is sampled from the standard normal distribution with $sp=0.1$.}
	\label{figure A=randn0.1}
\end{figure*}

\begin{figure*}[!t]
	\centering
	\subfloat[$m=200,n=100$]{\includegraphics[width=1.7in]{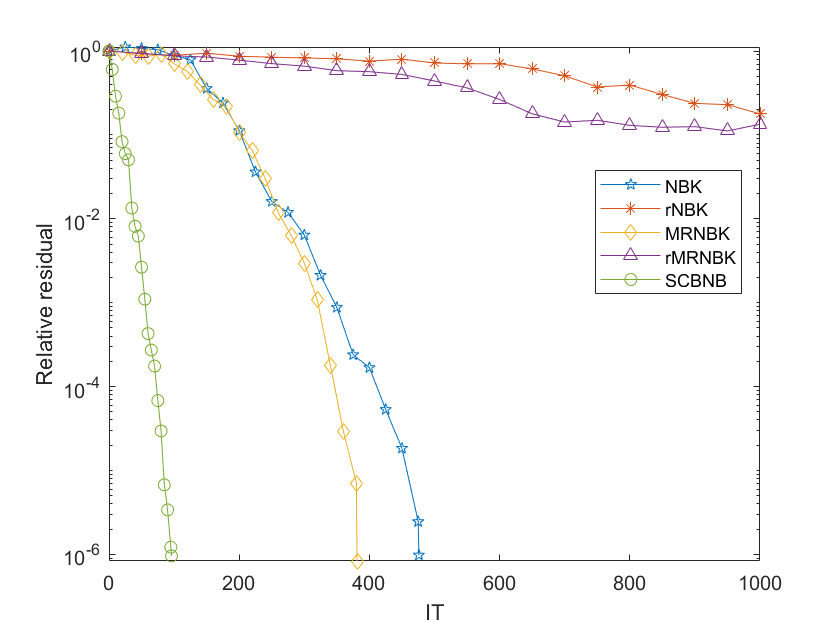}}
	\subfloat[$m=300,n=150$]{\includegraphics[width=1.7in]{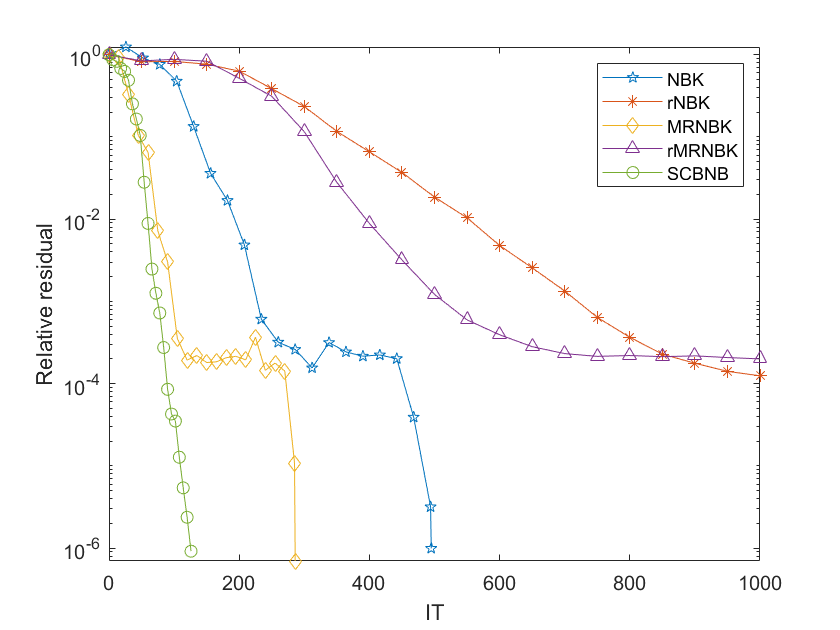}}
	\subfloat[$m=400,n=200$]{\includegraphics[width=1.7in]{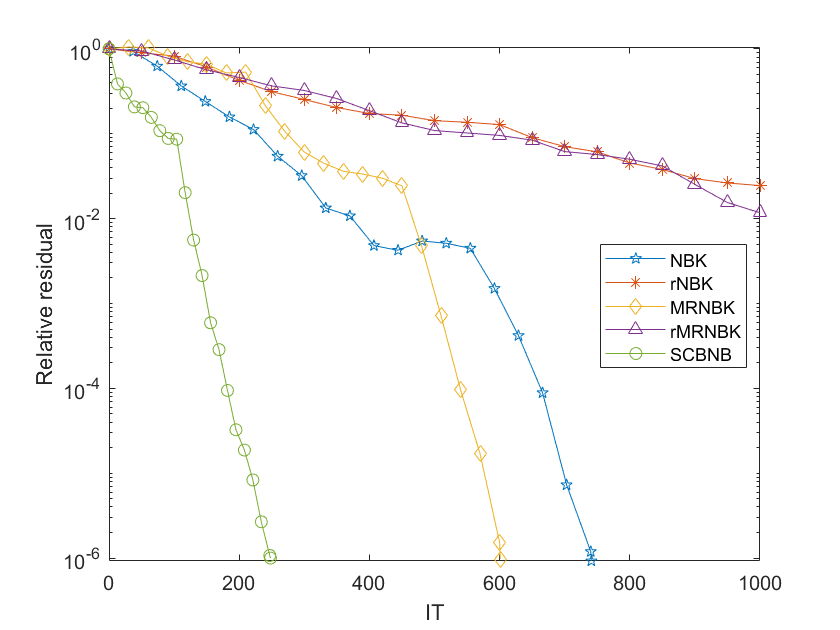}}
	\subfloat[$m=500,n=250$]{\includegraphics[width=1.7in]{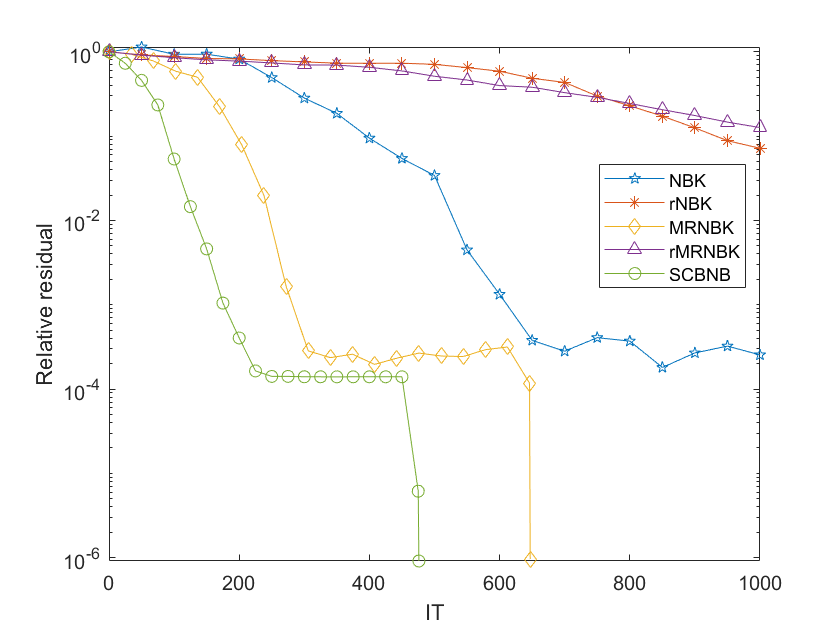}}
	\caption{Results for $A^{(i)}$ which is sampled from the standard normal distribution with $sp=0.05$.}
	\label{figure A=randn0.05}
\end{figure*}

\begin{figure*}[!t]
	\centering
	\subfloat[NBK]{\includegraphics[width=1.7in]{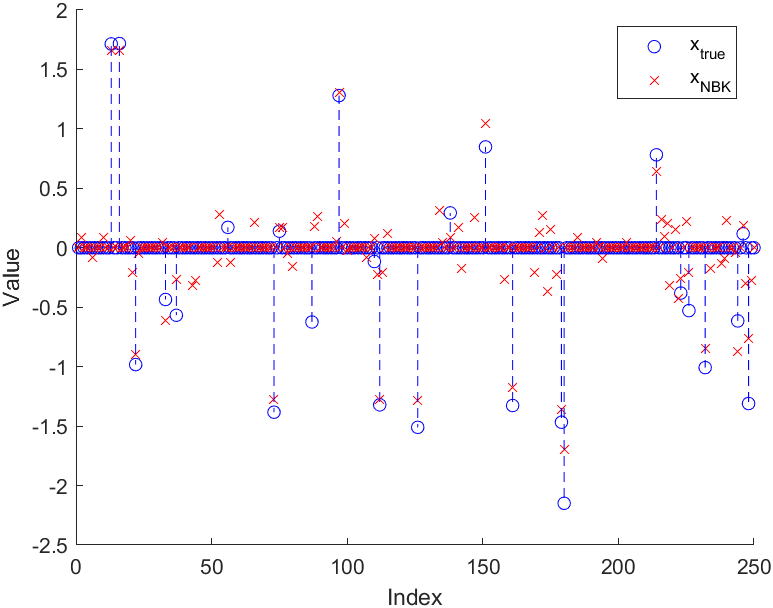}}
	\subfloat[rNBK]{\includegraphics[width=1.7in]{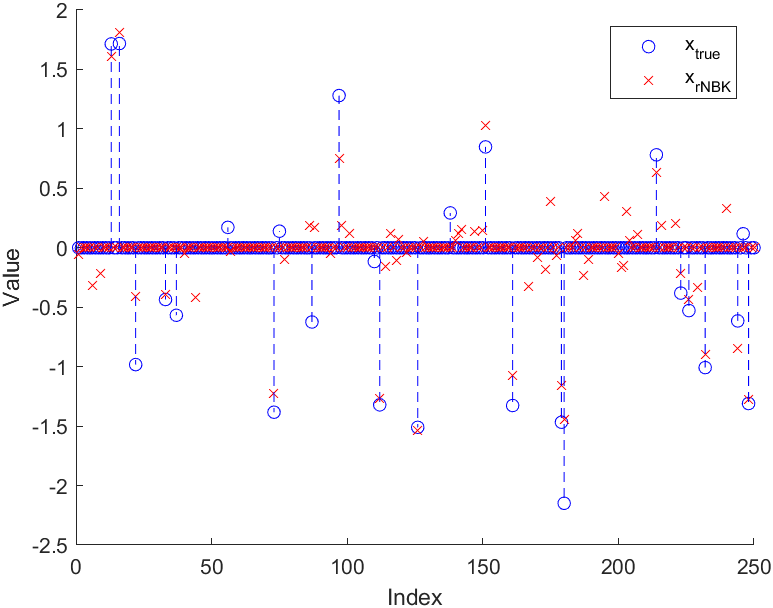}}
	\subfloat[MRNBK]{\includegraphics[width=1.7in]{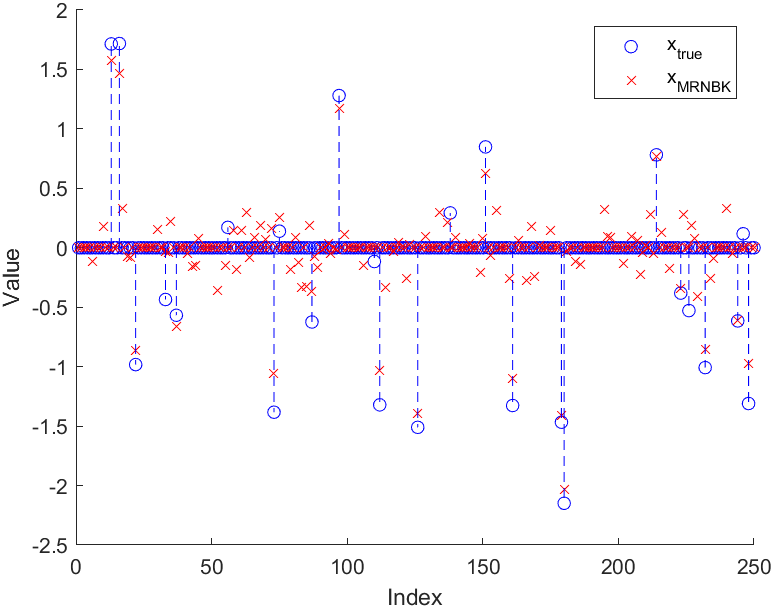}}\\
	\subfloat[rMRNBK]{\includegraphics[width=1.7in]{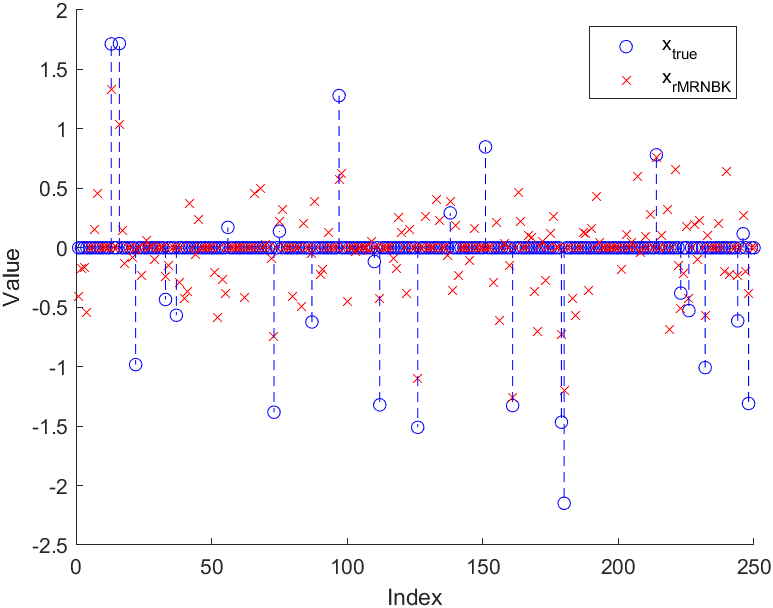}}
	\subfloat[SCBNB]{\includegraphics[width=1.7in]{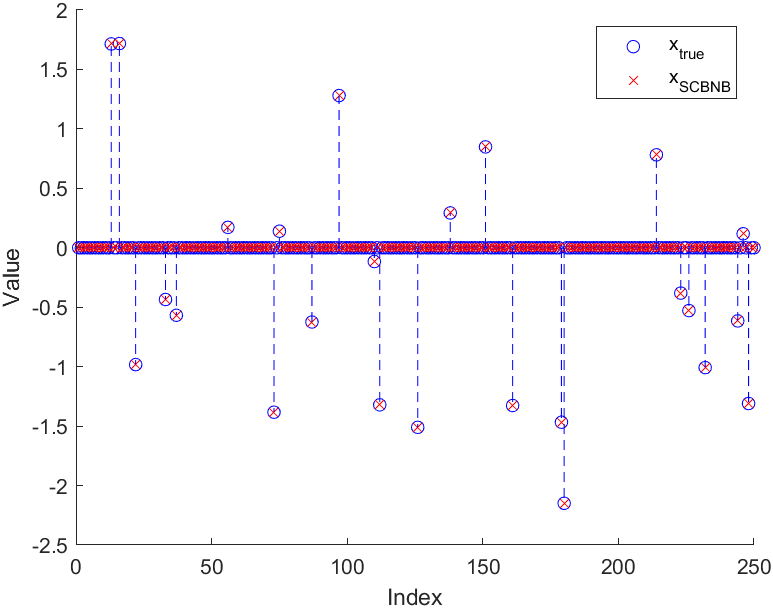}}
	\caption{The original signal and recovered signal of $A^{(i)}$ is sampled from the standard normal distribution with $m=500$, $n=250$ and $sp=0.1$.}
	\label{figure randn solution}
\end{figure*}

\begin{figure*}[!t]
	\centering
	\subfloat[NBK]{\includegraphics[width=1.7in]{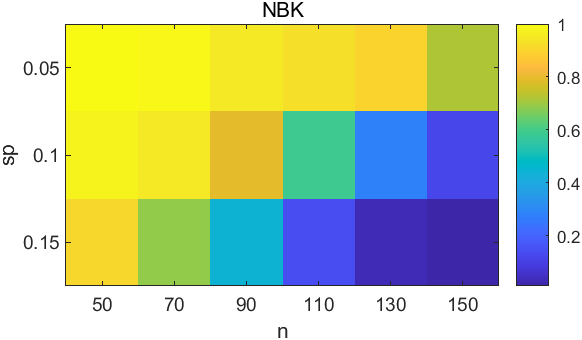}}
	\subfloat[rNBK]{\includegraphics[width=1.7in]{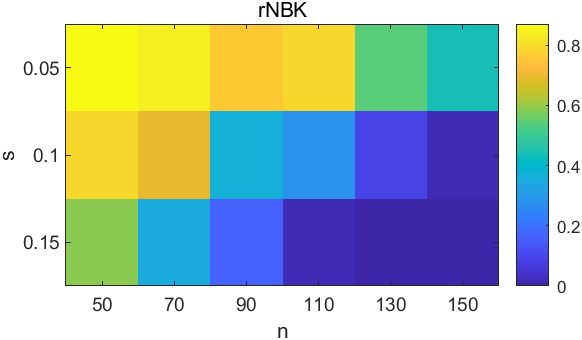}}
	\subfloat[MRNBK]{\includegraphics[width=1.7in]{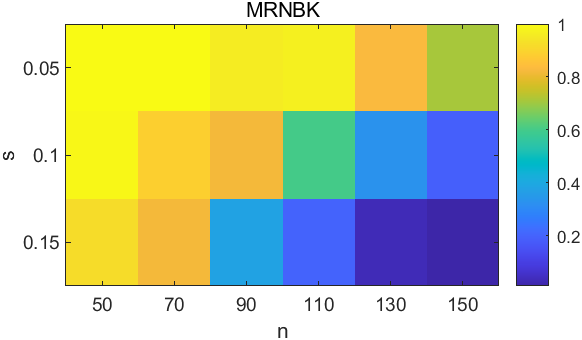}}\\
	\subfloat[rMRNBK]{\includegraphics[width=1.7in]{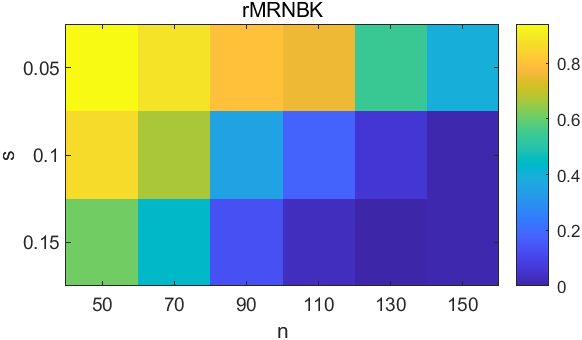}}
	\subfloat[SCBNB]{\includegraphics[width=1.7in]{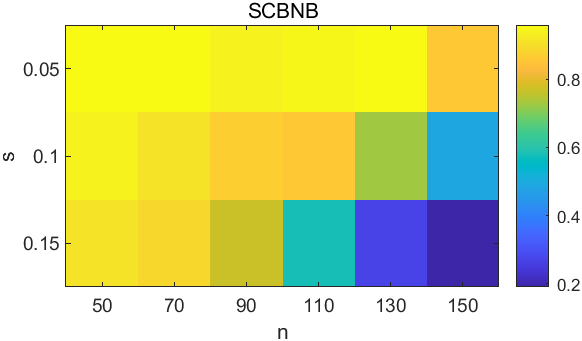}}
	\caption{Success rate of $A^{(i)}$ which is sampled from the standard normal distribution with $m=100$, $n=50:20:150$ and $sp=0.05:0.05:0.15$.}
	\label{figure randn success}
\end{figure*}

\begin{exm}
	$A^{(i)}$ is a random partial discrete cosine transform (DCT) matrix with j-th column is generated through the expression 
	\begin{equation*}
		A^{(i)}(:,j)=\cos(2\pi(j-1)\omega), \quad j=1,..,n,
	\end{equation*}
	where $\omega\in\mathbb{R}^{n\times1}$ is a column vector with uniformly and independently sampled elements from $[0,1]$.
\end{exm}

In Table \ref{tab:A=DCT0.1} and Table \ref{tab:A=DCT0.05}, we compare the IT and CPU of the new method against the existing methods on different dimensions of the problem for $sp=0.1$ and $sp=0.05$ of $\hat{x}$, respectively. Table \ref{tab:A=DCT0.1} and Table \ref{tab:A=DCT0.05} show that the SCBNB method outperforms the other methods. And Figure \ref{figure A=DCT0.1} and Figure \ref{figure A=DCT0.05} illustrate that as the dimension of the matrix increases, the SCBNB method still converges to the sparse solution at a faster rate, while the other methods slow down their convergence more significantly. In Figure \ref{figure DCT solution} we depict the original signal and the recovered signals by the five methods, respectively. We find the SCBNB method and MRNBK method have better recovery results. However, some of the signals recovered by the MRNBK method deviate from the original signal to a greater extent than those of the SCBNB method. Further, we also carried out the success rate of these methods for different $A$, $b$ and $\hat{x}$ settings. The basic settings of this experiment is same as in Example \ref{ex:rand}, with the difference being the generation of $A$. Figure \ref{figure DCT success} shows that the SCBNB method achieves better performance in terms of success rate compared with other methods, further validating the advantages of the SCBNB method.
\begin{table*}\label{tab:A=DCT0.1}
	\centering
	\caption{IT and CPU of the NBK, \textit{r}NBK, MRNBK, \textit{r}MRNBK and SCBNB methods for $A^{(i)}$ is a random partial DCT matrix with $sp=0.1$}
	\vspace{1mm}
	\begin{tabular}{cclllllllllll}
		\toprule
		$m$&  $n$& $sp$ & \multicolumn{2}{c}{NBK}& \multicolumn{2}{c}{rNBK}& \multicolumn{2}{c}{MRNBK}& \multicolumn{2}{c}{rMRNBK}& \multicolumn{2}{c}{SCBNB}\\
		& & &  IT&CPU& IT&  CPU& IT&CPU& IT& CPU& IT&CPU\\
		\midrule
		$200$&
		$100$& $0.1$& $460$&$0.5380$&$--$&$--$& $358$&$\mathbf{0.4106}$&$--$& $--$& $117$&$0.4383$\\
		$300$&$150$& $0.1$& $888$&$5.5576$&$--$&$--$& $769$&$4.7639$&$--$& $--$& $261$&$\mathbf{3.7346}$\\
		$400$&$200$& $0.1$& $--$&$--$&$--$&$--$& $--$&$--$&$--$& $--$& $454$&$\mathbf{14.2869}$\\
		$500$&$250$& $0.1$& $--$&$--$&$--$&$--$& $--$&$--$&$--$& $--$& $475$&$\mathbf{29.1530}$\\\bottomrule\end{tabular}
\end{table*}

\begin{table*}\label{tab:A=DCT0.05}
	\centering
	\caption{IT and CPU of the NBK, rNBK, MRNBK, rMRNBK and SCBNB methods for $A^{(i)}$ is a random partial DCT matrix with $sp=0.05$}
	\vspace{1mm}
	\begin{tabular}{cclllllllllll}
		\toprule
		$m$&  $n$& $sp$ & \multicolumn{2}{c}{NBK}& \multicolumn{2}{c}{rNBK}& \multicolumn{2}{c}{MRNBK}& \multicolumn{2}{c}{rMRNBK}& \multicolumn{2}{c}{SCBNB}\\
		& & &  IT&CPU& IT&  CPU& IT&CPU& IT& CPU& IT&CPU\\
		\midrule
		$200$&
		$100$& $0.05$& $118$&$0.1420$&$660$&$0.7446$& $92$&$\mathbf{0.1061}$&$363$& $0.3971$& $53$&$0.2132$\\
		$300$&$150$& $0.05$& $819$&$5.2926$&$--$&$--$& $--$&$--$&$--$& $--$& $180$&$\mathbf{2.5549}$\\
		$400$&$200$& $0.05$& $542$&$10.8815$&$--$&$--$& $388$&$9.0115$&$--$& $--$& $229$&$\mathbf{7.2557}$\\
		$500$&$250$& $0.05$& $846$&$35.9033$&$--$&$--$& $--$&$--$&$--$& $--$& $296$&$\mathbf{17.9961}$\\\bottomrule\end{tabular}
\end{table*}

\begin{figure*}[!t]
	\centering
	\subfloat[$m=200,n=100$]{\includegraphics[width=1.7in]{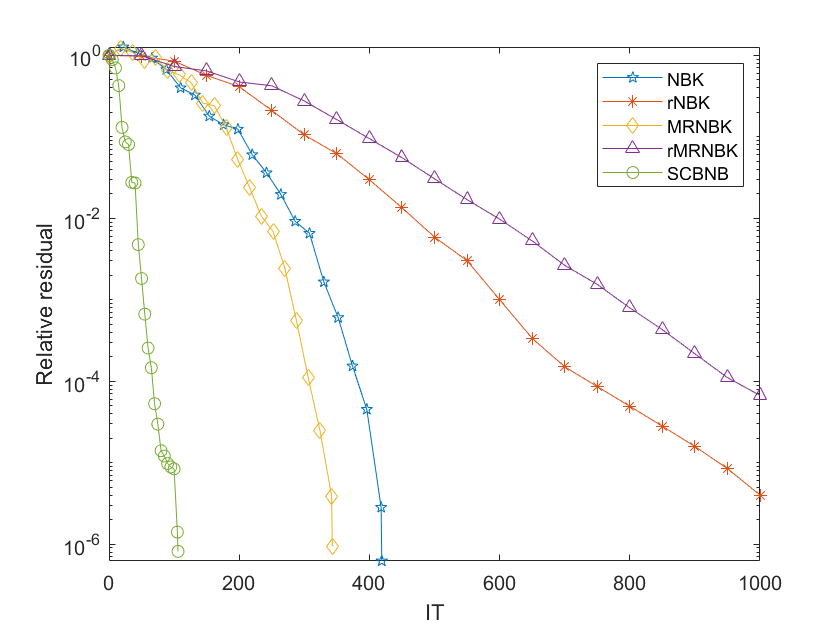}}
	\subfloat[$m=300,n=150$]{\includegraphics[width=1.7in]{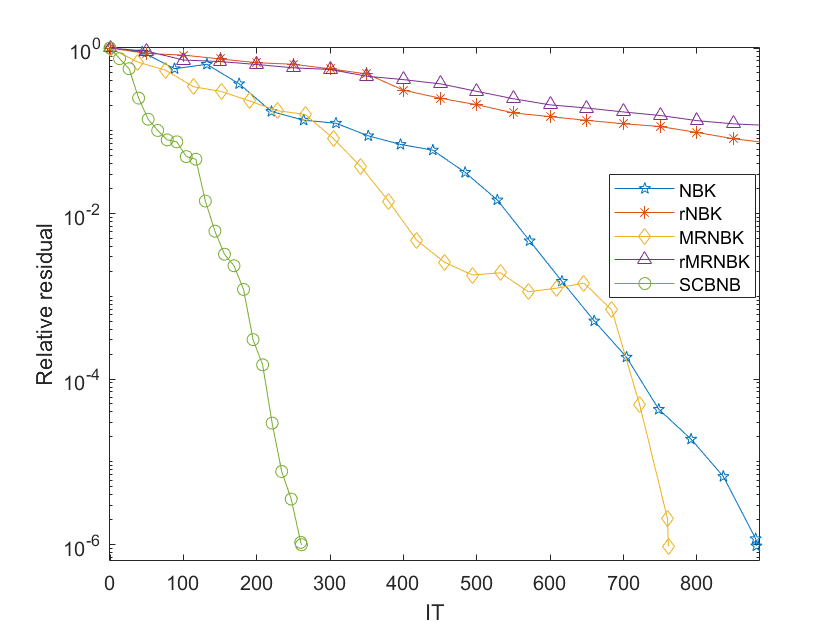}}
	\subfloat[$m=400,n=200$]{\includegraphics[width=1.7in]{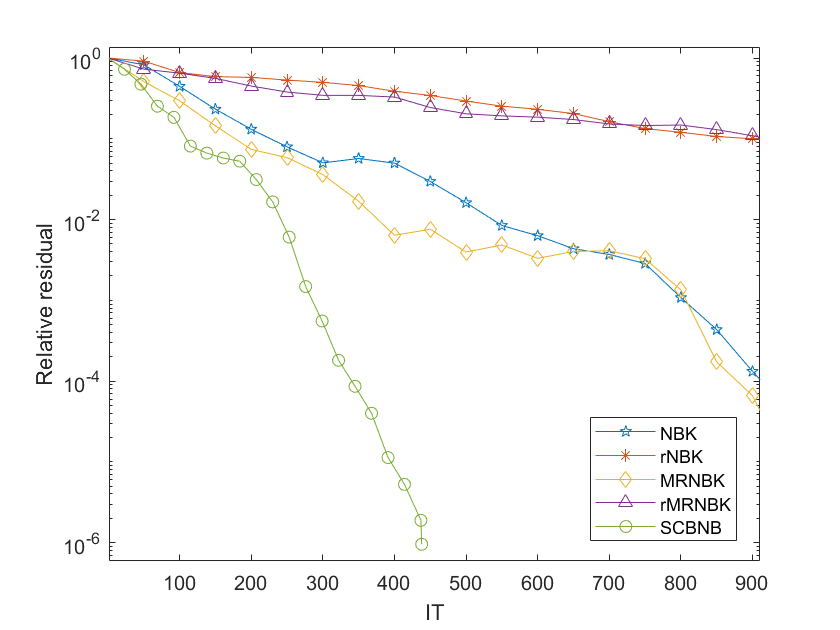}}
	\subfloat[$m=500,n=250$]{\includegraphics[width=1.7in]{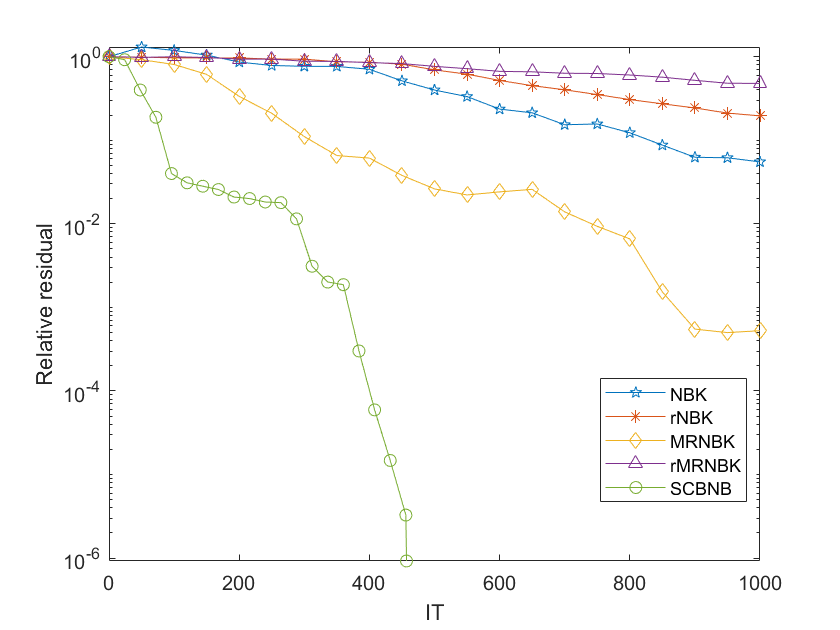}}
	\caption{Results for $A^{(i)}$ which is a random partial DCT matrix with $sp=0.1$.}
	\label{figure A=DCT0.1}
\end{figure*}

\begin{figure*}[!t]
	\centering
	\subfloat[$m=200,n=100$]{\includegraphics[width=1.7in]{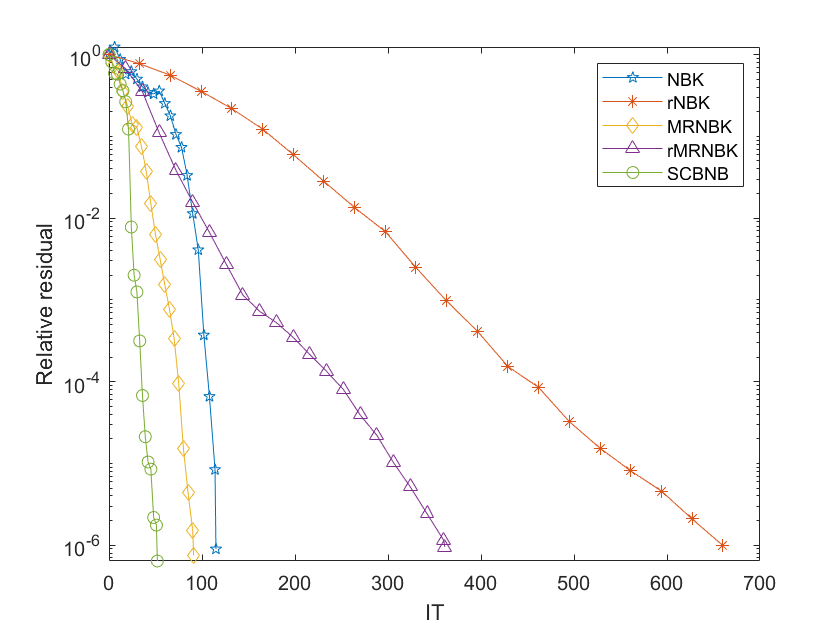}}
	\subfloat[$m=300,n=150$]{\includegraphics[width=1.7in]{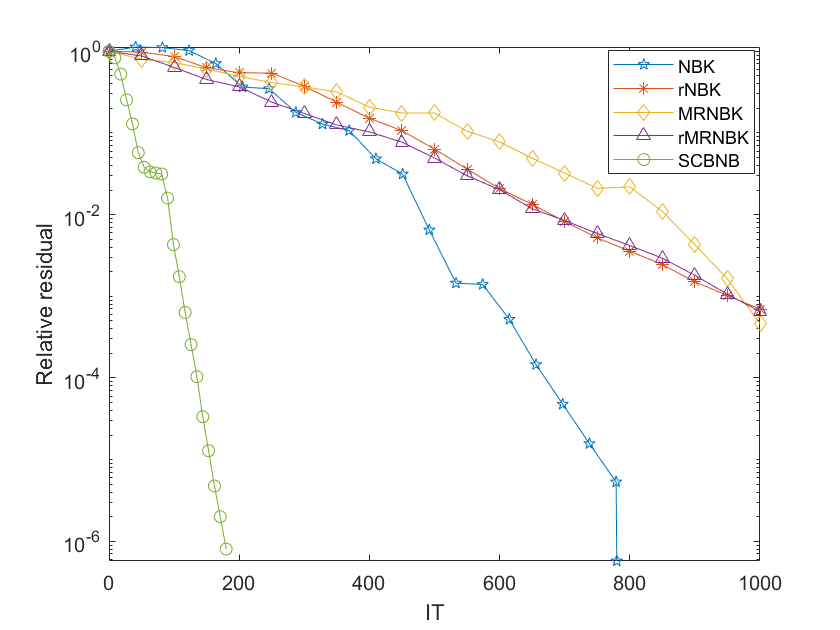}}
	\subfloat[$m=400,n=200$]{\includegraphics[width=1.7in]{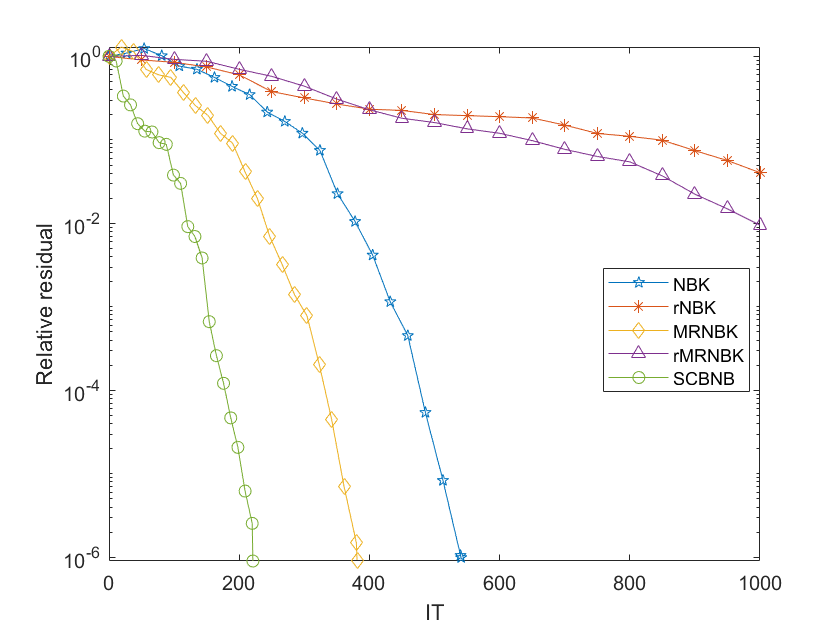}}
	\subfloat[$m=500,n=250$]{\includegraphics[width=1.7in]{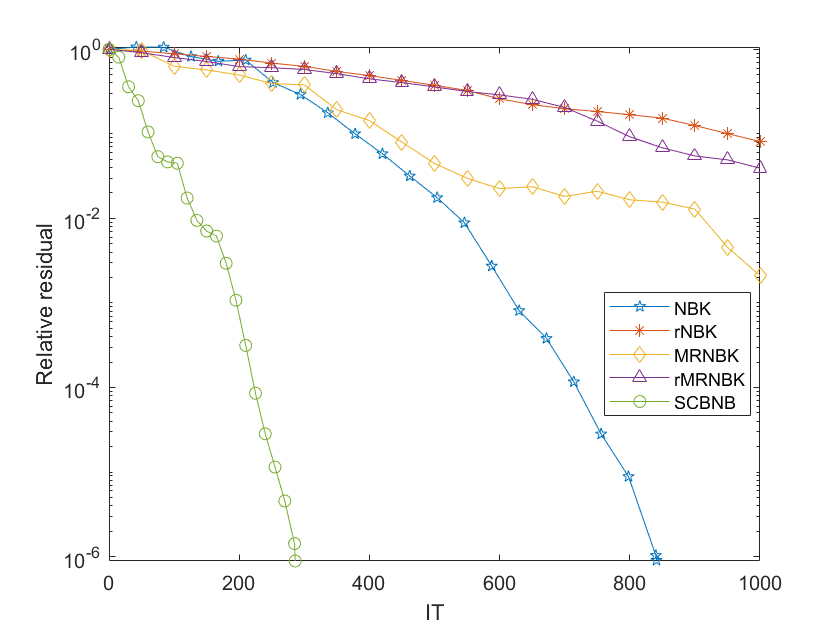}}
	\caption{Results for $A^{(i)}$ is a random partial DCT matrix with $sp=0.05$.}
	\label{figure A=DCT0.05}
\end{figure*}

\begin{figure*}[!t]
	\centering
	\subfloat[NBK]{\includegraphics[width=1.7in]{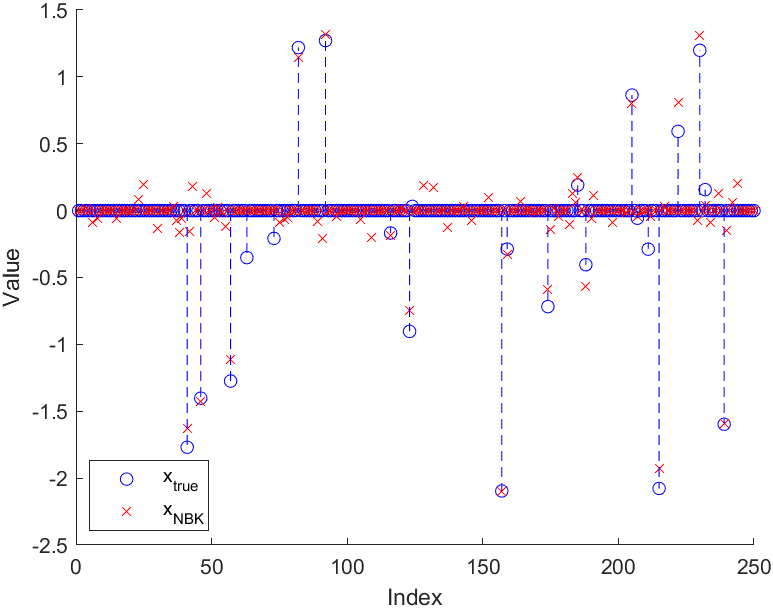}}
	\subfloat[rNBK]{\includegraphics[width=1.7in]{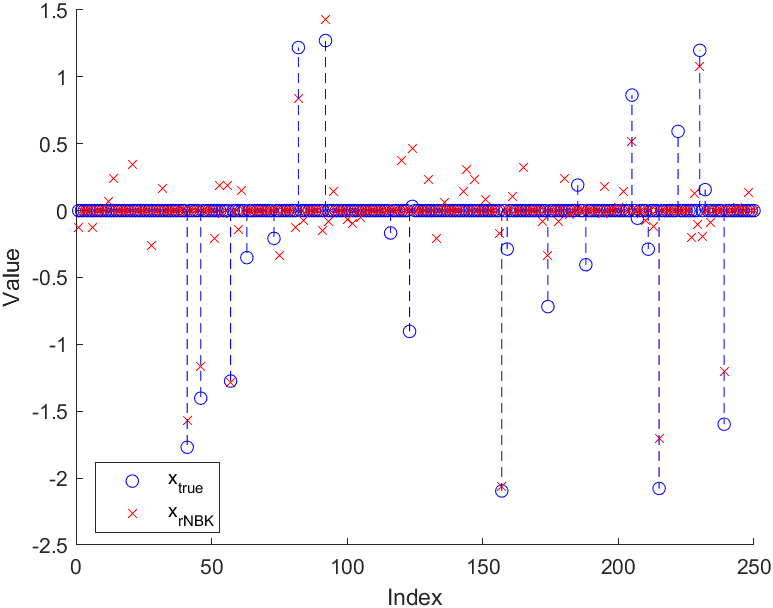}}
	\subfloat[MRNBK]{\includegraphics[width=1.7in]{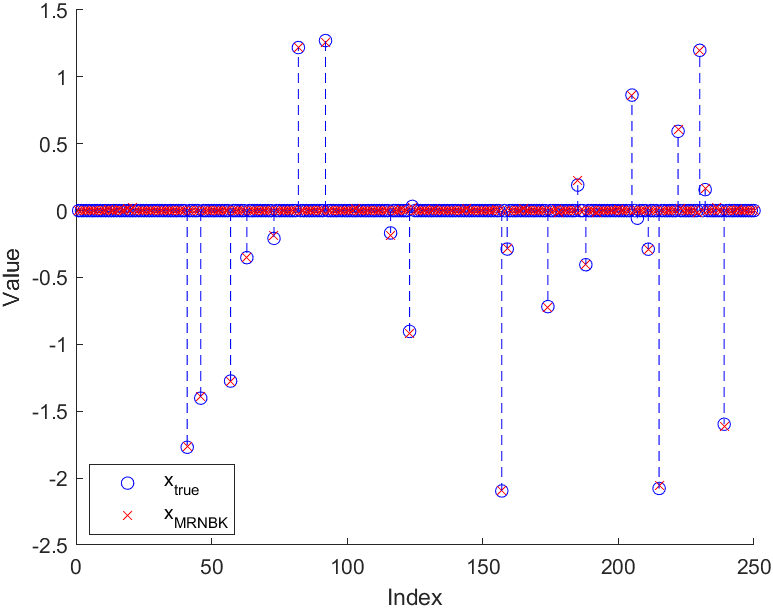}}\\
	\subfloat[rMRNBK]{\includegraphics[width=1.7in]{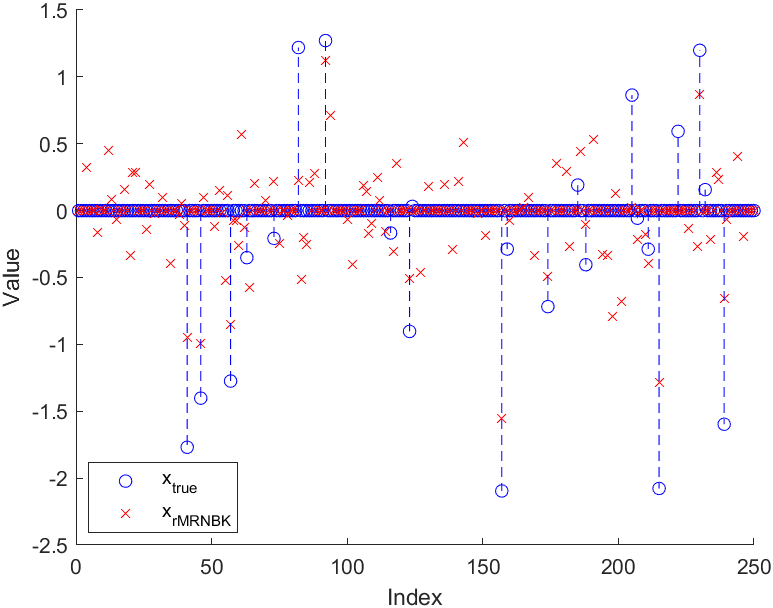}}
	\subfloat[SCBNB]{\includegraphics[width=1.7in]{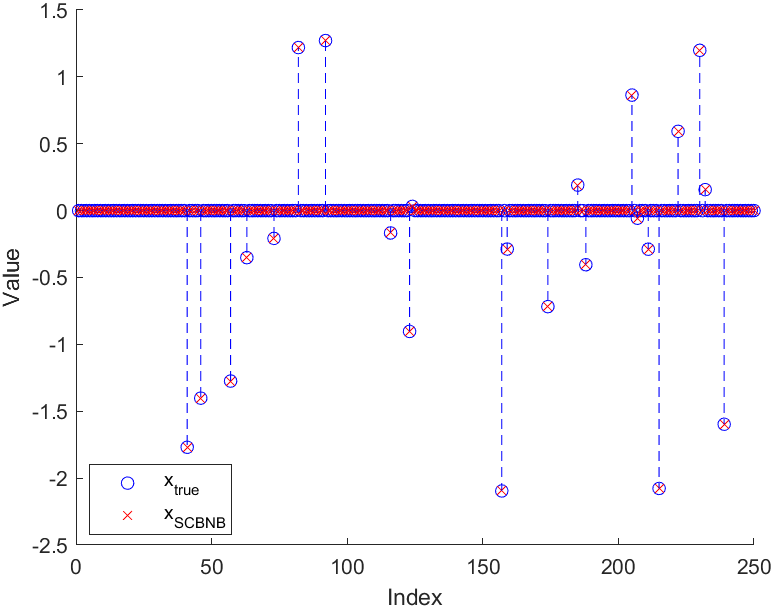}}
	\caption{The original signal and recovered signal of $A^{(i)}$ which is a random partial DCT matrix with $m=500$, $n=250$ and $sp=0.1$.}
	\label{figure DCT solution}
\end{figure*}

\begin{figure*}[!t]
	\centering
	\subfloat[NBK]{\includegraphics[width=1.7in]{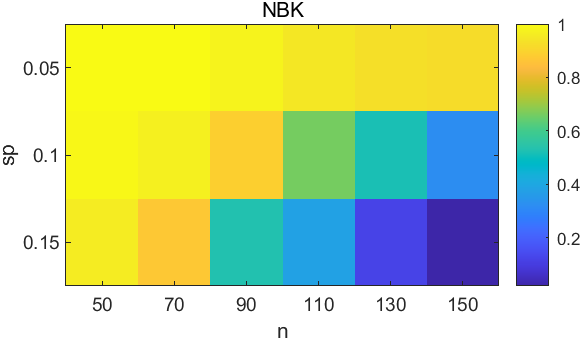}}
	\subfloat[rNBK]{\includegraphics[width=1.7in]{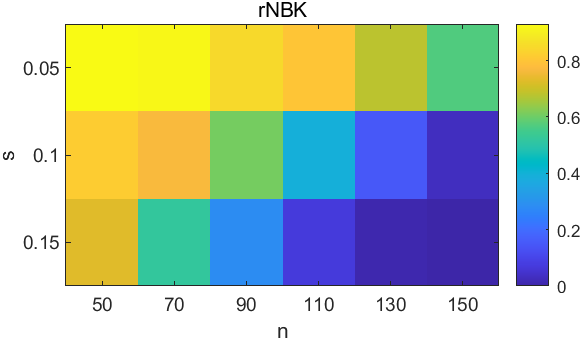}}
	\subfloat[MRNBK]{\includegraphics[width=1.7in]{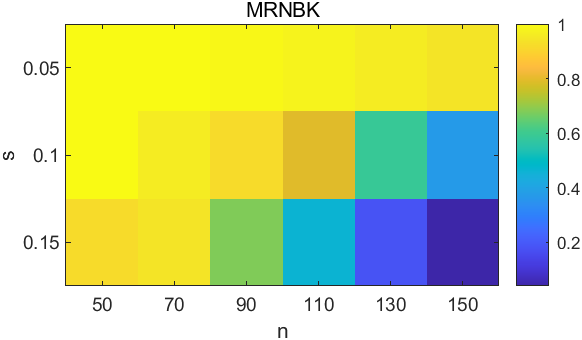}}\\
	\subfloat[rMRNBK]{\includegraphics[width=1.7in]{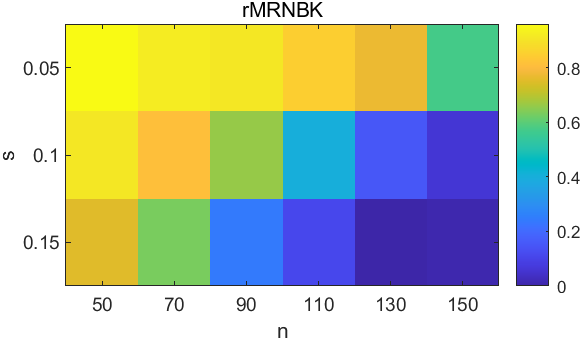}}
	\subfloat[SCBNB]{\includegraphics[width=1.7in]{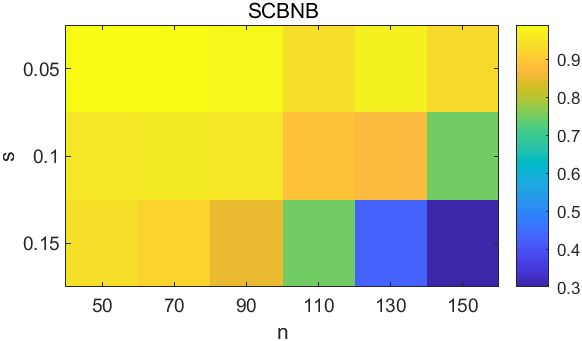}}
	\caption{Success rate of $A^{(i)}$ which is a random partial DCT matrix with $m=100$, $n=50:20:150$ and $sp=0.05:0.05:0.15$.}
	\label{figure DCT success}
\end{figure*}

\begin{exm}
	Spare image recovery
\end{exm}
In this part, we consider to recover the Mnist datasets problem. The original signal is the an $28\times28$ image $X_t$ in the Mnist dataset. Let the coefficient matrix $A$ is generated by the Matlab function rand(m,n) with $m=2000$ and $n=28^2$. Reshape the original image matrix to get the vector $\hat{x}$. The right-hand side vector is $b=A\hat{x}+e$, where $e$ is the Gaussian noise vector with the noise level of $0.01$. We use the peak signal-to-noise(PSNR) to measure the quality of the recovered image, which is defined as 
\begin{equation*}
	PSNR(X_t,X_r)=10*log_{10}\frac{\max(X_t(i,j))^2}{\frac{1/N^2}{\sum_{i=1}^{N}\sum_{j=1}^{N}|X_t(i,j)-X_r(i,j)|^2}},
\end{equation*} 
where $X_r$ is the recovered image.
Before the formal experiment, the original image needs to be normalized using the Matlab function $normalize(X_t,\text{`}range\text{'})$. We also use  $\varphi(x)=\lambda\|x\|_1+\frac{1}{2}\|x\|_2^2$ and set $\lambda=1$, the initial subgradient $x_0^*$ is a zero vector. Figure \ref{figure image} shows the results obtained from 1000 iterations of MRNBK, rMRNBK, and SCBNB methods, and the results show that the SCBNB method is the most effective.

\begin{figure*}[!t]
	\centering
	\subfloat[Original image]{\includegraphics[width=1.2in]{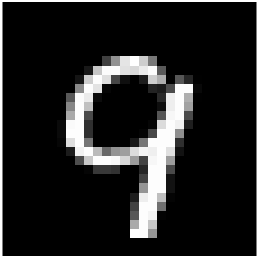}}\quad\quad\quad
	\subfloat[MRNBK]{\includegraphics[width=1.2in]{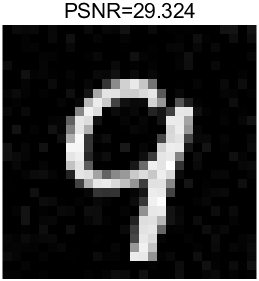}}\quad\quad\quad
	\subfloat[rMRNBK]{\includegraphics[width=1.2in]{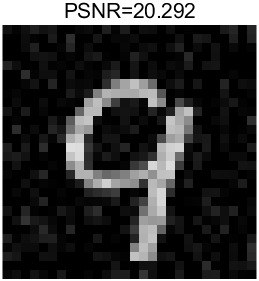}}\quad\quad\quad
	\subfloat[SCBNB]{\includegraphics[width=1.2in]{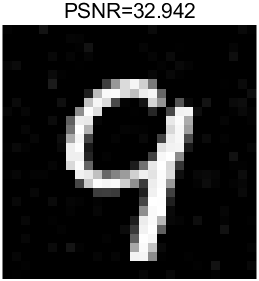}}
	\caption{Original image and recovered images of Mnist datas problem.}
	\label{figure image}
\end{figure*}

\section{Conclusion}
\label{sec4}

In this paper, we proposed the SCBNB method, which can find the sparse solution of nonlinear equations. Under certain assumptions we derived the convergence property of the SCBNB method. In the numerical experiments the IT and CPU of the SCBNB method is less than the NBK, rNBK, MRNBK and rMRNBK methods, which shows the efficiency of the new method. And other methods become very slow to converge when the matrix dimension is increased, but the SCBNB method still has an impressive convergence rate. In future, we will investigate the strategy of how to choose the block sizes.  

%
%
%
%
%
%
%
%
%
%
\bibliographystyle{IEEEtran}
\bibliography{reference1}

{\appendix[Proof of Theorem \ref{thm1}]
	
}	From Lemma \ref{lem:D}, we have
{
\begin{align}\label{eq:Dk}
	&\resizebox{0.28\hsize}{!}{$\quad D_{\varphi}^{x_{k+1}^*}(x_{k+1},\hat{x})$}\nonumber\\
	&\resizebox{0.85\hsize}{!}{$\leq D_{\varphi}^{x_{k}^*}(x_{k},\hat{x})+\langle x_{k+1}^*-x_{k}^*,x_k-\hat{x}\rangle+\frac{1}{2\gamma}\|x_{k+1}^*-x_{k}^*\|_2^2$}\nonumber\\ 
	&\resizebox{0.22\hsize}{!}{$=D_{\varphi}^{x_{k}^*}(x_{k},\hat{x})$}\nonumber\\ 
	&\quad\resizebox{0.85\hsize}{!}{$+\langle-\delta\frac{\gamma\|\nabla f_{:,\xi_k}(x_k)^Tf(x_k)\|_2^2}{\|\nabla f_{:,\xi_k}(x_k)\nabla f_{:,\xi_k}(x_k)^Tf(x_k)\|_2^2}I_{:,\xi_k}\nabla f_{:,\xi_k}(x_k)^Tf(x_k),x_k-\hat{x} \rangle$}\nonumber\\ \nonumber
	&\quad\resizebox{0.85\hsize}{!}{$+\frac{1}{2\gamma}\|-\delta\frac{\gamma\|\nabla f_{:,\xi_k}(x_k)^Tf(x_k)\|_2^2}{\|\nabla f_{:,\xi_k}(x_k)\nabla f_{:,\xi_k}(x_k)^Tf(x_k)\|_2^2}I_{:,\xi_k}\nabla f_{:,\xi_k}(x_k)^Tf(x_k)\|_2^2$}\nonumber\\ 
	&\resizebox{0.22\hsize}{!}{$= D_{\varphi}^{x_{k}^*}(x_{k},\hat{x})$}\nonumber\\
	&\quad\resizebox{0.85\hsize}{!}{$-\delta\frac{\gamma\|\nabla f_{:,\xi_k}(x_k)^Tf(x_k)\|_2^2}{\|\nabla f_{:,\xi_k}(x_k)\nabla f_{:,\xi_k}(x_k)^Tf(x_k)\|_2^2}f(x_k)^T\nabla f_{:,\xi_k}(x_k)I_{:,\xi_k}^T(x_k-\hat{x})$}.\nonumber\\ 
	&\quad\resizebox{0.45\hsize}{!}{$+\frac{\gamma\delta^2}{2}\frac{\|\nabla f_{:,\xi_k}(x_k)^Tf(x_k)\|_2^6}{\|\nabla f_{:,\xi_k}(x_k)\nabla f_{:,\xi_k}(x_k)^Tf(x_k)\|_2^4}$}.
\end{align}
}
By Lemma \ref{lem:sigma} we have

\begin{align*}
	&\|\nabla f_{:,\xi_k}(x_k)\nabla f_{:,\xi_k}(x_k)^Tf(x_k)\|_2\\
	&\leq\sigma_{\max}(\nabla f_{:,\xi_k}(x_k))\|\nabla f_{:,\xi_k}(x_k)^Tf(x_k)\|_2\\ 
	&\leq\sigma_{\max}(\nabla f(x_k))\|\nabla f_{:,\xi_k}(x_k)^Tf(x_k)\|_2\\
	&\leq\overline{\sigma}\|\nabla f_{:,\xi_k}(x_k)^Tf(x_k)\|_2,
\end{align*}
and 
\begin{align*}
	&\|\nabla f_{:,\xi_k}(x_k)\nabla f_{:,\xi_k}(x_k)^Tf(x_k)\|_2\\
	&\geq\sigma_{\min}(\nabla f_{:,\xi_k}(x_k))\|\nabla f_{:,\xi_k}(x_k)^Tf(x_k)\|_2\\
	&\geq\sigma_{\min}(\nabla f(x_k))\|\nabla f_{:,\xi_k}(x_k)^Tf(x_k)\|_2\\
	&\geq\underline{\sigma}\|\nabla f_{:,\xi_k}(x_k)^Tf(x_k)\|_2.
\end{align*}
 
Taking the conditional expectation on inequality (\ref{eq:Dk}), we can have
\begin{align*}
	&\resizebox{0.28\hsize}{!}{$\mathbf{E}_k[D_{\varphi}^{x_{k+1}^*}(x_{k+1},\hat{x})]$}\\
	&\resizebox{0.75\hsize}{!}{$\leq D_{\varphi}^{x_{k}^*}(x_{k},\hat{x})+\frac{\delta^2\gamma}{2\underline{\sigma}^4\tau}\sum_{\xi_k}\|\nabla f_{:,\xi_k}(x_k)^Tf(x_k)\|_2^2$}\\
	&\quad\resizebox{0.9\hsize}{!}{$-\frac{\delta\gamma}{\tau}\sum_{\xi_k}\frac{\|\nabla f_{:,\xi_k}(x_k)^Tf(x_k)\|_2^2}{\|\nabla f_{:,\xi_k}(x_k)\nabla f_{:,\xi_k}(x_k)^Tf(x_k)\|_2^2}\left(f(x_k)^T\nabla f_{:,\xi_k}(x_k)I_{:,\xi_k}^T(x_k-\hat{x})\right)$}\\
	&\resizebox{0.75\hsize}{!}{$= D_{\varphi}^{x_{k}^*}(x_{k},\hat{x})+\frac{\delta^2\gamma}{2\underline{\sigma}^4\tau}\sum_{\xi_k}\|\nabla f_{:,\xi_k}(x_k)^Tf(x_k)\|_2^2$}\\
	&\quad\resizebox{0.65\hsize}{!}{$-\frac{\delta\gamma}{\tau}\left[\frac{1}{\overline{\sigma}^2}\sum_{\xi_k^+}\left(f(x_k)^T\nabla f_{:,\xi_k}(x_k)I_{:,\xi_k}^T(x_k-\hat{x})\right)\right .$}\\
	&\quad\resizebox{0.65\hsize}{!}{$\left .+\frac{1}{\underline{\sigma}^2}\sum_{\xi_k^-}\left(f(x_k)^T\nabla f_{:,\xi_k}(x_k)I_{:,\xi_k}^T(x_k-\hat{x})\right)\right]$}\\
	&\resizebox{0.6\hsize}{!}{$=D_{\varphi}^{x_{k}^*}(x_{k},\hat{x})+\frac{\delta^2\gamma}{2\underline{\sigma}^4\tau}\|\nabla f(x_k)^Tf(x_k)\|_2^2$}\\
	&\quad\resizebox{0.65\hsize}{!}{$-\frac{\delta\gamma}{\tau}\left[\frac{1}{\overline{\sigma}^2}\sum_{\xi_k}\left(f(x_k)^T\nabla f_{:,\xi_k}(x_k)I_{:,\xi_k}^T(x_k-\hat{x})\right)\right .$}\\
	&\quad\resizebox{0.8\hsize}{!}{$\left .+\left(\frac{1}{\underline{\sigma}^2}-\frac{1}{\overline{\sigma}^2}\right)\sum_{\xi_k^-}\left(f(x_k)^T\nabla f_{:,\xi_k}(x_k)I_{:,\xi_k}^T(x_k-\hat{x})\right)\right]$}\\
	&\leq\resizebox{0.9\hsize}{!}{$ D_{\varphi}^{x_{k}^*}(x_{k},\hat{x})-\frac{\gamma\delta}{\tau\overline{\sigma}^2}f(x_k)^T\nabla f(x_k)(x_k-\hat{x})+\frac{\delta^2\gamma}{2\underline{\sigma}^4\tau}\|\nabla f(x_k)^Tf(x_k)\|_2^2$}\\
	&\quad\resizebox{0.8\hsize}{!}{$+\frac{\delta\gamma}{\tau}\left(\frac{1}{\underline{\sigma}^2}-\frac{1}{\overline{\sigma}^2}\right)\sum_{\xi_k^-}\|f(x_k)^T\nabla f_{:,\xi_k}(x_k)\|\|x_k-\hat{x}\|$}\\
	&\quad\resizebox{0.8\hsize}{!}{$+\frac{\delta\gamma}{\tau}\left(\frac{1}{\underline{\sigma}^2}-\frac{1}{\overline{\sigma}^2}\right)\sum_{\xi_k^+}\|f(x_k)^T\nabla f_{:,\xi_k}(x_k)\|\|x_k-\hat{x}\|$}\\
	&=\resizebox{0.8\hsize}{!}{$D_{\varphi}^{x_{k}^*}(x_{k},\hat{x})+\frac{\gamma\delta}{2\tau}\left(\frac{2f(x_k)^T}{\overline{\sigma}^2}(f(x_k)-f(\hat{x})-\nabla f(x_k)(x_k)-\hat{x})\right .$}\\
	&\quad\resizebox{0.6\hsize}{!}{$\left .+\frac{\delta}{\underline{\sigma}^4}\|\nabla f(x_k)^Tf(x_k)\|_2^2-\frac{2\|f(x_k)\|_2^2}{\overline{\sigma}^2}\right)$}\\
	&\resizebox{0.65\hsize}{!}{$\quad+\frac{\delta\gamma}{\tau}\left(\frac{1}{\underline{\sigma}^2}-\frac{1}{\overline{\sigma}^2}\right)\|f(x_k)^T\nabla f(x_k)\|\|x_k-\hat{x}\|$}\\
	&\leq \resizebox{0.9\hsize}{!}{$D_{\varphi}^{x_{k}^*}(x_{k},\hat{x})+\frac{\gamma\delta}{2\tau}\left(\frac{2\|f(x_k)\|_2}{\overline{\sigma}^2}\|f(x_k)-f(\hat{x})-\nabla f(x_k)(x_k)-\hat{x}\|_2\right .$}\\
	&\resizebox{0.45\hsize}{!}{$\quad\left .+\left(\frac{\delta\overline{\sigma}^2}{\underline{\sigma}^4}-\frac{2}{\overline{\sigma}^2}\right)\|f(x_k)\|_2^2\right)$}\\
	&\resizebox{0.65\hsize}{!}{$\quad+\frac{\gamma\delta}{\tau}\left(\frac{1}{\underline{\sigma}^2}-\frac{1}{\overline{\sigma}^2}\right)\|f(x_k)^T\nabla f(x_k)\|\|x_k-\hat{x}\|$}\\
	&=A+B.
\end{align*}
,where $\sum_{\xi_k^+}\left(f(x_k)^T\nabla f_{:,\xi_k}(x_k)I_{:,\xi_k}^T(x_k-\hat{x})\right)$ represents the sum of all $f(x_k)^T\nabla f_{:,\xi_k}(x_k)I_{:,\xi_k}^T(x_k-\hat{x})>0$, and $\sum_{\xi_k^-}\left(f(x_k)^T\nabla f_{:,\xi_k}(x_k)I_{:,\xi_k}^T(x_k-\hat{x})\right)$ represents the sum of all $f(x_k)^T\nabla f_{:,\xi_k}(x_k)I_{:,\xi_k}^T(x_k-\hat{x})<0$.

By Assumption \ref{asm1} ii), iv) and Lemma \ref{lem:1+eta^2}, we can obtain
\begin{align*}
	A
	&\leq D_{\varphi}^{x_{k}^*}(x_{k},\hat{x})+\frac{\gamma\delta}{2\tau}\left(\frac{2\eta\|f(x_k)\|^2_2}{\overline{\sigma}^2}+\left(\frac{\delta\overline{\sigma}^2}{\underline{\sigma}^4}-\frac{2}{\overline{\sigma}^2}\right)\|f(x_k)\|_2^2\right)\\
	&=D_{\varphi}^{x_{k}^*}(x_{k},\hat{x})-\frac{\gamma}{2\tau}\left(\frac{2(1-\eta)\delta}{\overline{\sigma}^2}-\frac{\overline{\sigma}^2\delta^2}{\underline{\sigma}^4}\right)\|f(x_k)\|^2_2\\
	&\leq\resizebox{1\hsize}{!}{$ D_{\varphi}^{x_{k}^*}(x_{k},\hat{x})-\frac{\gamma}{2\tau}\left(\frac{2(1-\eta)\delta}{\overline{\sigma}^2}-\frac{\overline{\sigma}^2\delta^2}{\underline{\sigma}^4}\right)\frac{1}{(1+\eta)^2}\|\nabla f(\hat{x})(x_k-\hat{x})\|_2^2$}\\
	&\leq \resizebox{1\hsize}{!}{$D_{\varphi}^{x_{k}^*}(x_{k},\hat{x})-\frac{\gamma}{2\tau}\left(\frac{2(1-\eta)\delta}{\overline{\sigma}^2}-\frac{\overline{\sigma}^2\delta^2}{\underline{\sigma}^4}\right)\frac{1}{(1+\eta)^2}\sigma_{\min}^2(\nabla f(\hat{x}))\|x_k-\hat{x}\|_2^2$}\\
	&\leq D_{\varphi}^{x_{k}^*}(x_{k},\hat{x})-\frac{\gamma}{2\tau}\left(\frac{2(1-\eta)\delta}{\overline{\sigma}^2}-\frac{\overline{\sigma}^2\delta^2}{\underline{\sigma}^4}\right)\frac{1}{(1+\eta)^2}\underline{\sigma}^2\|x_k-\hat{x}\|_2^2\\
\end{align*} 
\begin{align*}
	B
	&\leq\frac{\gamma\delta}{\tau}\left(\frac{1}{\underline{\sigma}^2}-\frac{1}{\overline{\sigma}^2}\right)\overline{\sigma}\|f(x_k)\|\|(x_k-\hat{x})\|\\
	&\leq\frac{\gamma\delta(\overline{\sigma}^2-\underline{\sigma}^2)}{\tau\underline{\sigma}^2(1-\eta)}\|x_k-\hat{x}\|_2^2.
\end{align*}
By Lemma \ref{lem:D<=(x*-y*)(x-y)}, we can obtain
\begin{equation*}
	D_{\varphi}^{x_k^*}(x_k,\hat{x})\leq \langle x_k^*-\hat{x}^*,x_k-\hat{x}\rangle.
\end{equation*}
Combining $\varphi(x)$ is $M$-smooth and Lemma \ref{lem:smooth}, we can get
\begin{equation*}
	D_{\varphi}^{x_k^*}(x_k,\hat{x})\leq M\|x_k-\hat{x}\|_2^2.
\end{equation*}
Then we can get 
\begin{align*}
	&\mathbf{E}_k[D_{\varphi}^{x_{k+1}^*}(x_{k+1},\hat{x})]\\
	&\leq\resizebox{1\hsize}{!}{$ D_{\varphi}^{x_{k}^*}(x_{k},\hat{x})-\frac{\gamma}{2\tau}\left(\frac{2(1-\eta)\delta}{\overline{\sigma}^2}-\frac{\overline{\sigma}^2\delta^2}{\underline{\sigma}^4}\right)\frac{1}{(1+\eta)^2}\underline{\sigma}^2\|(x_k-\hat{x})\|_2^2+\frac{\gamma\delta(\overline{\sigma}^2-\underline{\sigma}^2)}{\tau\underline{\sigma}^2(1-\eta)}\|x_k-\hat{x}\|_2^2$}\\
	&=\resizebox{1\hsize}{!}{$D_{\varphi}^{x_{k}^*}(x_{k},\hat{x})-\left[\left(\frac{\gamma\underline{\sigma}^2(1-\eta)}{\tau(1+\eta)^2\overline{\sigma}^2}-\frac{\gamma(\overline{\sigma}^2-\underline{\sigma}^2)}{\tau\underline{\sigma}^2(1-\eta)}\right)\delta-\frac{\gamma\overline{\sigma}^2}{2\tau(1+\eta)^2\underline{\sigma}^2}\delta^2\right]\|x_k-\hat{x}\|_2^2$}\\
	&\leq\resizebox{1\hsize}{!}{$ D_{\varphi}^{x_{k}^*}(x_{k},\hat{x})-\frac{1}{M}\left[\frac{2\gamma\underline{\sigma}^4(1-\eta)^2-2\gamma(\overline{\sigma}^2-\underline{\sigma}^2)(1+\eta)^2\overline{\sigma}^2}{2\tau(1+\eta)^2\overline{\sigma}^2\underline{\sigma}^2(1-\eta)}\delta-\frac{\gamma\underline{\sigma}^4(1-\eta)}{2\tau(1+\eta)^2\overline{\sigma}^2\underline{\sigma}^2(1-\eta)}\delta^2\right]D_{\varphi}^{x_{k}^*}(x_{k},\hat{x})$}\\
	&=(1-c)D_{\varphi}^{x_{k}^*}(x_{k},\hat{x}).
\end{align*}
Taking the full expectation of both sides, we can conclude (\ref{eq: D_k<D_0}).
By Lemma \ref{lem:D<=(x*-y*)(x-y)}, we have
\begin{equation*}
	\frac{\gamma}{2}\|x_k-\hat{x}\|_2^2\leq D_\varphi^{x^*}(x_k,\hat{x}).
\end{equation*}
If $x_0=x_0^*=0$ we can obtain $D_\varphi^{x_0^*}(x_0,\hat{x})=\varphi(\hat{x})$.
Then we can get (\ref{eq: x_k-x*}).
\begin{rem}
	To guarantee $c>0$, we use the quadratic property of $\delta$, leading to the bound $0<\delta<2\frac{\underline{\sigma}^4(1-\eta)^2-(\overline{\sigma}^2-\underline{\sigma}^2)(1+\eta)^2\overline{\sigma}^2}{\overline{\sigma}^4(1-\eta)}$ under the prerequisite $\underline{\sigma}^4(1-\eta)^2-(\overline{\sigma}^2-\underline{\sigma}^2)(1+\eta)^2\overline{\sigma}^2>0$. This thus requires ensuring that
	\begin{equation*}
		(\underline{\sigma}^4-\overline{\sigma}^4+\underline{\sigma}^2\overline{\sigma}^2)\eta^2+2(-\underline{\sigma}^4-\overline{\sigma}^4+\underline{\sigma}^2\overline{\sigma}^2)\eta+(\underline{\sigma}^4-\overline{\sigma}^4+\underline{\sigma}^2\overline{\sigma}^2)>0.
	\end{equation*}
	Let $a=\underline{\sigma}^4$, $b=-\overline{\sigma}^4+\underline{\sigma}^2\overline{\sigma}^2$ and substituting them into the above inequality, we obtain
	\begin{equation}\label{ineq:eta}
		g(\eta)=  (a+b)\eta^2+2(-a+b)\eta+(a+b)>0.
	\end{equation}
	It is easy to see that $a>0$ and $b\leq0$. The discriminant of $g(\eta)$ is 
	\begin{equation*}
		\Delta=4(a-b)^2-4(a+b)^2=-16ab\geq0.
	\end{equation*}
	
	i) When $a+b=0$, $g(\eta)=2(-a+b)\eta<0$. In this case there exists no interval for $\eta$ that satisfies the required condition.
	
	ii) When $a+b\neq0$ and $b=0$ which is equivalent to $\overline{\sigma}=\underline{\sigma}$, we can obtain $\Delta=0$. The root of $g(\eta)=0$ is $\eta=1$.
	By Assumption \ref{asm1} ii), we know that $0<\eta<1$. 
	
	If $a+b>0$, then the interval satisfying inequality (\ref{ineq:eta}) is $0<\eta<1$. 
	
	If $a+b<0$, then there exists no interval for $\eta$ that satisfies the required condition.
	
	iii) When $a+b\neq0$ and $b\neq0$, we can obtain $\Delta>0$. The root of $g(\eta)=0$ is $\eta_1=\frac{a-b}{a+b}+2\frac{\sqrt{|ab|}}{a+b}$, $\eta_2=\frac{a-b}{a+b}-2\frac{\sqrt{|ab|}}{a+b}$.
	
	If $a+b>0$, then the interval satisfying inequality (\ref{ineq:eta}) is $\eta\in(-\infty,\frac{a-b}{a+b}-2\frac{\sqrt{|ab|}}{a+b})\cup(\frac{a-b}{a+b}+2\frac{\sqrt{|ab|}}{a+b},+\infty)$. Since $0<\eta<1$, $\frac{a-b}{a+b}+2\frac{\sqrt{|ab|}}{a+b}>1$ and $\frac{a-b}{a+b}-2\frac{\sqrt{|ab|}}{a+b}=\frac{(\sqrt{|a|}-\sqrt{|b|})^2}{a+b}>0$, the range of $\eta$ is 
	\begin{equation*}
		0<\eta<\min\{1,\frac{a-b}{a+b}-2\frac{\sqrt{|ab|}}{a+b}\}.
	\end{equation*}
	
	If $a+b<0$, then the interval satisfying inequality (\ref{ineq:eta}) is $\eta\in(\frac{a-b}{a+b}+2\frac{\sqrt{|ab|}}{a+b},\frac{a-b}{a+b}-2\frac{\sqrt{|ab|}}{a+b})$. It is easy to see that $a-b>0$, $\sqrt{|ab|}>0$, then $\frac{a-b}{a+b}+2\frac{\sqrt{|ab|}}{a+b}<0$ and $\frac{a-b}{a+b}-2\frac{\sqrt{|ab|}}{a+b}=\frac{(\sqrt{|a|}-\sqrt{|b|})^2}{a+b}<0$. Since $0<\eta<1$, there exists no interval for $\eta$ that satisfies the required condition.
	
	In summary, only when $a+b>0$ is there a chance to satisfy the condition. When $b=0$, there is $0<\eta<1$. When $b\neq0$, there is $0<\eta<\min\{1,\frac{a-b}{a+b}-2\frac{\sqrt{|ab|}}{a+b}\}$ satisfy inequality (\ref{ineq:eta}).
\end{rem}
\end{document}